\documentclass{amsart}
\usepackage{amsmath}
\usepackage{amssymb}
\usepackage{amsthm}

\DeclareMathOperator{\tr}{tr}
\DeclareMathOperator{\Tr}{Tr}

\DeclareMathOperator{\dvol}{dvol}

\DeclareMathOperator{\Ric}{Ric}

\DeclareMathOperator{\Dom}{Dom}

\newcommand{\oJ}{\overline{J}}

\newcommand{\oP}{\overline{P}}

\newcommand{\oX}{\overline{X}}

\newcommand{\oDelta}{\overline{\Delta}}

\newcommand{\onabla}{\overline{\nabla}}

\newcommand{\hg}{\widehat{g}}
\newcommand{\hr}{\widehat{r}}

\newcommand{\hB}{\widehat{B}}

\newcommand{\hP}{\widehat{P}}

\newcommand{\hT}{\widehat{T}}
\newcommand{\hmB}{\widehat{\mathcal{B}}}

\newcommand{\hnabla}{\widehat{\nabla}}
\newcommand{\hdelta}{\widehat{\delta}}

\newcommand{\heta}{\widehat{\eta}}
\newcommand{\hrho}{\widehat{\rho}}

\newcommand{\lp}{\langle}
\newcommand{\rp}{\rangle}
\newcommand{\lv}{\lvert}
\newcommand{\rv}{\rvert}
\newcommand{\lV}{\lVert}
\newcommand{\rV}{\rVert}

\newcommand{\mB}{\mathcal{B}}
\newcommand{\mC}{\mathcal{C}}
\newcommand{\mD}{\mathcal{D}}
\newcommand{\mE}{\mathcal{E}}

\newcommand{\mH}{\mathcal{H}}

\newcommand{\mP}{\mathcal{P}}
\newcommand{\mQ}{\mathcal{Q}}

\newcommand{\kB}{\mathfrak{B}}

\newcommand{\bN}{\mathbb{N}}

\newcommand{\bR}{\mathbb{R}}

\newcommand{\suchthat}{\mathrel{}\middle|\mathrel{}}

\def\sideremark#1{\ifvmode\leavevmode\fi\vadjust{\vbox to0pt{\vss
 \hbox to 0pt{\hskip\hsize\hskip1em
 \vbox{\hsize3cm\tiny\raggedright\pretolerance10000
 \noindent #1\hfill}\hss}\vbox to8pt{\vfil}\vss}}}

\newcommand{\comment}[1]{}

\newtheorem{thm}{Theorem}[section]
\newtheorem{prop}[thm]{Proposition}
\newtheorem{lem}[thm]{Lemma}
\newtheorem{cor}[thm]{Corollary}

\theoremstyle{definition}
\newtheorem{defn}[thm]{Definition}

\theoremstyle{remark}

\numberwithin{equation}{section}

\begin{document}

\title{Some energy inequalities involving fractional GJMS operators}
\author{Jeffrey S. Case}
\address{109 McAllister Building \\ Penn State University \\ University Park, PA 16802}
\email{jscase@psu.edu}
\begin{abstract}
Under a spectral assumption on the Laplacian of a Poincar\'e--Einstein manifold, we establish an energy inequality relating the energy of a fractional GJMS operator of order $2\gamma\in(0,2)$ or $2\gamma\in(2,4)$ and the energy of the weighted conformal Laplacian or weighted Paneitz operator, respectively.  This spectral assumption is necessary and sufficient for such an inequality to hold.  We prove the energy inequalities by introducing conformally covariant boundary operators associated to the weighted conformal Laplacian and weighted Paneitz operator which generalize the Robin operator.  As an application, we establish a new sharp weighted Sobolev trace inequality on the upper hemisphere.
\end{abstract}
\maketitle

\section{Introduction}
\label{sec:intro}

Fractional GJMS operators are conformally covariant pseudodifferential operators defined on the boundary of a Poincar\'e--Einstein manifold via scattering theory which have principal symbol equal to that of the fractional powers of the Laplacian~\cite{GrahamZworski2003}.  Fractional GJMS operators can also be understood as generalized Dirichlet-to-Neumann operators associated to weighted GJMS operators of a suitable order defined in the interior~\cite{BransonGover2001,CaffarelliSilvestre2007,CaseChang2013,ChangGonzalez2011,Yang2013}.  In particular, one can identify the energy associated to a fractional GJMS operator with the energy associated to a suitable weighted GJMS operator when restricted to canonical extensions; see~\cite{CaffarelliSilvestre2007,Yang2013} for the flat case and~\cite{CaseChang2013,ChangGonzalez2011} for the curved case.

In this article, we are interested in obtaining, as a generalization of known results in the flat case~\cite{Yang2013}, a general relationship between the energy associated to a fractional GJMS operator and the energy associated to a suitable weighted GJMS operator for arbitrary extensions.  One reason for this interest is the role of such relationships in establishing sharp Sobolev trace inequalities (cf.\ \cite{AcheChang2015,Escobar1988}) and in studying the fractional Yamabe problem (cf.\ \cite{Escobar1992a,GonzalezQing2010}).  Indeed, this article is partly motivated by a subtle issue which arises in the works of Escobar~\cite{Escobar1992a,Escobar1992ac} and Gonzalez--Qing~\cite{GonzalezQing2010} on the fractional Yamabe problem of order $\gamma\in(0,1)$.  In both works, one tries to find a metric on a compact manifold with boundary which is scalar flat in the interior and for which the boundary has constant mean curvature (in a sense made precise in Section~\ref{sec:boundary}) by minimizing an energy functional in the interior subject to a volume-normalization on the boundary.  However, there is no guarantee that the energy functional is bounded below within this class, an issue overlooked in~\cite{Escobar1992a,GonzalezQing2010} and corrected in the special case $\gamma=1/2$ in~\cite{Escobar1992ac}.  Proposition~\ref{prop:bounded_below} corrects this issue by giving a spectral condition under which the energy functional is bounded below.

The main results of this article are the following two theorems.  These results establish, under spectral assumptions on a Poincar\'e--Einstein manifold, energy inequalities on suitable compactifications of the Poincar\'e--Einstein manifold which relate the energy of the weighted conformal Laplacian and the weighted Paneitz operator to the energy of the fractional GJMS operators $P_{2\gamma}$ in the cases $\gamma\in(0,1)$ and $\gamma\in(1,2)$, respectively.  That equality holds for the special extensions $U$ was established by Chang and the author~\cite{CaseChang2013}.

\begin{thm}
 \label{thm:intro_inequality}
 Fix $\gamma\in(0,1)$ and set $m=1-2\gamma$.  Let $(X^{n+1},M^n,g_+)$ be a Poincar\'e--Einstein manifold satisfying $\lambda_1(-\Delta_{g_+})>\frac{n^2}{4}-\gamma^2$.  Let $r$ be a geodesic defining function for $M$ and let $\rho$ be a defining function such that, asymptotically near $M$,
 \[ \rho = r + \Phi r^{1+2\gamma} + o(r^{1+2\gamma}) \]
 for some $\Phi\in C^\infty(M)$.  Fix $f\in C^\infty(M)$ and denote by $\mD_f^\gamma$ the set of functions $U\in C^\infty(X)\cap C^0(\oX)$ such that, asymptotically near $M$,
 \[ U = f + \psi\rho^{2\gamma} + o(\rho^{2\gamma}) \]
 for some $\psi\in C^\infty(M)$.  Set $g=\rho^2g_+$ and $h=g\rv_{TM}$.  Then
 \begin{multline}
  \label{eqn:intro_inequality}
  \int_X \left(\lv\nabla U\rv^2 + \frac{m+n-1}{2}J_\phi^m U^2\right) \rho^m\,\dvol_g \\ \geq -\frac{2\gamma}{d_\gamma}\left[ \oint_M f\,P_{2\gamma}f\,\dvol_h - \frac{n-2\gamma}{2}d_\gamma\oint_M\Phi f^2\dvol_h \right]
 \end{multline}
 for all $U\in\mD_f^\gamma$, where $J_\phi^m$ is the weighted scalar curvature of $(\oX,g,\rho,m,1)$.  Moreover, equality holds if and only if $L_{2,\phi}^mU=0$.
\end{thm}

Note that the left-hand side of~\eqref{eqn:intro_inequality} is the Dirichlet energy of the weighted conformal Laplacian $L_{2,\phi}^m$ of $(\oX,g,\rho,m,1)$.  See Section~\ref{sec:bg} for a detailed explanation of the terminology and notation used in Theorem~\ref{thm:intro_inequality}.  The spectral condition in Theorem~\ref{thm:intro_inequality} holds for Poincar\'e--Einstein manifolds for which the conformal infinity $(M^n,[h])$ has nonnegative Yamabe constant~\cite{Lee1995}.

A key point is that the spectral assumption $\lambda_1\left(-\Delta_{g_+}\right)>\frac{n^2}{4}-\gamma^2$ is necessary; see Proposition~\ref{prop:bounded_below}.  This corrects the aforementioned mistake in~\cite{GonzalezQing2010}.  Observe also that the left-hand side of~\eqref{eqn:intro_inequality} involves the interior $L^2$-norm of $U$.  This contrasts with the sharp Sobolev trace inequalities of Jin and Xiong~\cite{JinXiong2012} which instead involve a boundary $L^2$-norm of $f=U\rv_M$: Given a Poincar\'e--Einstein manifold $(X^{n+1},M^n,g_+)$, a constant $\gamma\in(0,1)$, and a defining function $\rho$ as in Theorem~\ref{thm:intro_inequality}, there is a constant $A$ such that
\begin{equation}
 \label{eqn:jx}
 \int_X \lv\nabla U\rv^2\rho^{1-2\gamma}\dvol_g + A\oint_M f^2\dvol \geq S(n,\gamma)\left(\oint_M \lv f\rv^{\frac{2n}{n-2\gamma}}\right)^{\frac{n-2\gamma}{n}}
\end{equation}
for any $U\in \mD^\gamma:=\bigcup_f\mD_f^\gamma$, where $g=\rho^2g_+$, $f=U\rv_M$, and $S(n,\gamma)$ is the corresponding constant in the upper half space~\cite{GonzalezQing2010,JinXiong2012}.  Under the spectral assumption $\lambda_1(-\Delta_{g_+})>\frac{n^2}{4}-\gamma^2$, one can use the adapted defining function~\cite[Subsection~6.1]{CaseChang2013} in Theorem~\ref{thm:intro_inequality} to eliminate the interior $L^2$-norm of $U$; indeed, combining this with~\eqref{eqn:jx} yields the sharp fractional Sobolev inequality
\[ \oint_M f\,P_{2\gamma}f + A\oint_M f^2 \geq -\frac{d_\gamma}{2\gamma}S(n,\gamma)\left(\oint_M \lv f\rv^{\frac{2n}{n-2\gamma}}\right)^{\frac{n-2\gamma}{n}} \]
for all $f\in C^\infty(M)$ (cf.\ \cite{HebeyVaugon1996,JinXiong2012}).

\begin{thm}
\label{thm:intro_inequality2}
 Fix $\gamma\in(1,2)$ and set $m=3-2\gamma$.  Let $(X^{n+1},M^n,g_+)$ be a Poincar\'e--Einstein manifold satisfying $\lambda_1(-\Delta_{g_+})>\frac{n^2}{4}-(2-\gamma)^2$.  Let $r$ be a geodesic defining function for $M$ and let $\rho$ be a defining function such that, asymptotically near $M$,
 \[ \rho = r + \rho_2 r^3 + \Phi r^{1+2\gamma} + o(r^{1+2\gamma}) \]
 for some $\rho_2,\Phi\in C^\infty(M)$.  Fix $f\in C^\infty(M)$ and denote by $\mD_f^\gamma$ the set of functions $U\in C^\infty(X)\cap C^0(\oX)$ such that, asymptotically near $M$,
 \[ U = f + f_2\rho^2 + \psi\rho^{2\gamma} + o(\rho^{2\gamma}) \]
 for some $f_2,\psi\in C^\infty(M)$.  Set $g=\rho^2g_+$ and $h=g\rv_{TM}$.  Then for any $U\in\mD_f^\gamma$ it holds that
 \begin{multline}
  \label{eqn:intro_inequality2}
  \int_X\left(\left(\Delta_\phi U\right)^2 - \left(4P-(n-2\gamma+2)J_\phi^mg\right)(\nabla U,\nabla U) + \frac{n-2\gamma}{2}Q_\phi^m U^2\right) \\
  \geq \frac{8\gamma(\gamma-1)}{d_\gamma}\left(\oint_M f\,P_{2\gamma}f - \frac{n-2\gamma}{2}d_\gamma\oint_M \Phi f^2\right) ,
 \end{multline}
 where $P$ is the Schouten tensor of $g$, $J_\phi^m$ and $Q_\phi^m$ are the weighted scalar curvature and the weighted $Q$-curvature, respectively, of $(\oX,g,\rho,m,1)$, and integrals on $X$ and $M$ are evaluated with respect to $\rho^m\dvol_g$ and $\dvol_h$, respectively.  Moreover, equality holds if and only if $L_{4,\phi}^mU=0$.
\end{thm}

Note that the left-hand side of~\eqref{eqn:intro_inequality2} is the Dirichlet energy of the weighted Paneitz operator $L_{4,\phi}^m$ of $(\oX,g,\rho,m,1)$.  See Section~\ref{sec:bg} for a detailed explanation of the terminology and notation used in Theorem~\ref{thm:intro_inequality}.  The spectral condition in Theorem~\ref{thm:intro_inequality2} holds for Poincar\'e--Einstein manifolds for which the conformal infinity $(M^n,[h])$ has nonnegative Yamabe constant~\cite{Lee1995}.

The proofs of Theorem~\ref{thm:intro_inequality} and Theorem~\ref{thm:intro_inequality2} rely on three observations.  First, we introduce conformally boundary covariant operators associated to the weighted conformal Laplacian and the weighted Paneitz operator in the same sense as the trace and Robin operators act as boundary operators associated to the conformal Laplacian (cf.\ \cite{Branson1997,BransonGover2001,Escobar1990,Escobar1992a}).  Second, we show that our conformally covariant operators recover certain scattering operators when acting on functions which lie in the kernel of the corresponding weighted GJMS operator on a Poincar\'e--Einstein manifold; this yields another approach to defining the fractional GJMS operators via extensions (cf.\ \cite{AcheChang2015,CaseChang2013,ChangGonzalez2011,GrahamZworski2003,GuillarmouGuillope2007}).  Third, using conformal covariance, we characterize when the left-hand sides of~\eqref{eqn:intro_inequality} and~\eqref{eqn:intro_inequality2} are uniformly bounded below in terms of spectral data for the metric $g_+$.  When these spectral conditions are met, the left-hand sides of~\eqref{eqn:intro_inequality} and~\eqref{eqn:intro_inequality2} can be minimized, and the identification of the minimizers follows from our extension theorem.

The second step in the above outline is a refinement of previous work of Chang and the author~\cite{CaseChang2013}.  In that work, it was shown that the fractional GJMS operators are generalized Dirichlet-to-Neumann operators for the weighted GJMS operators.  For example, under the assumptions of Theorem~\ref{thm:intro_inequality}, it was shown that if $L_{2,\phi}^mU=0$ and $U\rv_M=f$, then
\[ P_{2\gamma}f = -\frac{d_\gamma}{2\gamma}\lim_{\rho\to0}\left(\rho^m\eta U - \gamma(n-2\gamma)\Phi U\right); \]
see~\cite[Theorem~4.1]{CaseChang2013}.  In particular, equality holds in~\eqref{eqn:intro_inequality}.  The novelty introduced in this article is to realize the right-hand side of the above display as the evaluation of a conformally covariant boundary operator.  This also allows us to establish the energy inequality of Theorem~\ref{thm:intro_inequality}.  A similar comparison of our resuls to those of Chang and the author~\cite{CaseChang2013} holds in the case $\gamma\in(1,2)$.

As an application of our results, we establish a sharp Sobolev trace inequality on the standard upper hemisphere
\[ S_+^{n+1} := \left\{ x=(x_0,\dotsc, x_{n+1})\in\bR^{n+2} \colon x_{n+1}>0, \lv x\rv=1 \right\} \]
with the metric induced by the Euclidean metric.  To that end, let $\gamma\in(1,2)$ and set
\[ \mD^\gamma := \bigcup_{f\in C^\infty(S^n)} \mD_f^\gamma \]
for $\mD_f^\gamma$ determined by the defining function $x_{n+1}$ for $S^n=\partial S_+^{n+1}$ as in Theorem~\ref{thm:intro_inequality2}.

\begin{thm}
 \label{thm:sobolev}
 Fix $\gamma\in(1,2)$, choose $2\gamma<n\in\bN$, and let $(S_+^{n+1},d\theta^2)$ be the standard upper hemisphere.  Then
 \begin{multline}
  \label{eqn:sobolev}
  c_{n,\gamma}^{(2)}\left(\oint_{S^n} \lv f\rv^{\frac{2n}{n-2\gamma}}\dvol\right)^{\frac{n-2\gamma}{n}} \\ \leq \int_{S_+^{n+1}} \left[ \left(\Delta_\phi U\right)^2 + \frac{(n+3-2\gamma)^2-5}{2}\lv\nabla U\rv^2 + \frac{\Gamma\left(\frac{n+8-2\gamma}{2}\right)}{\Gamma\left(\frac{n-2\gamma}{2}\right)}U^2 \right] x_{n+1}^{3-2\gamma}\dvol
 \end{multline}
 for all $U\in\mD^\gamma$, where $f=U\rv_{S^n}$ and
 \[ c_{n,\gamma}^{(2)} = 8\pi^\gamma\frac{\Gamma(2-\gamma)}{\Gamma(\gamma)}\frac{\Gamma\left(\frac{n+2\gamma}{2}\right)}{\Gamma\left(\frac{n-2\gamma}{2}\right)}\left(\frac{\Gamma(n/2)}{\Gamma(n)}\right)^{\frac{2\gamma}{n}} . \]
 Moreover, equality holds if and only if
 \begin{equation}
  \label{eqn:paneitz_sobolev}
  \left(\Delta_\phi - \frac{(n+3-2\gamma)^2-1}{4}\right)\left(\Delta_\phi - \frac{(n+3-2\gamma)^2-9}{4}\right)U = 0
 \end{equation}
 and $f(x) = c\left(1+a\cdot x\right)^{-\frac{n-2\gamma}{2}}$ for some $c\in\bR$ and $a\in\bR^{n+1}$ with $\lv a\rv<1$.
\end{thm}

The corresponding result when $\gamma\in(0,1)$ is that
\[ c_{n,\gamma}^{(1)}\left(\oint_{S^n} \lv f\rv^{\frac{2n}{n-2\gamma}}\dvol\right)^{\frac{n-2\gamma}{n}} \leq \int_{S_+^{n+1}}\left[\lv\nabla U\rv^2 + \frac{\Gamma\left(\frac{n+4-2\gamma}{2}\right)}{\Gamma\left(\frac{n-2\gamma}{2}\right)}U^2\right]x_{n+1}^{1-2\gamma}\dvol \]
for all $U\in\mD^\gamma$ with trace $f=U\rv_{S^n}$, where
\[ c_{n,\gamma}^{(1)} = 2\pi^\gamma\frac{\Gamma(1-\gamma)}{\Gamma(\gamma)}\frac{\Gamma\left(\frac{n+2\gamma}{2}\right)}{\Gamma\left(\frac{n-2\gamma}{2}\right)}\left(\frac{\Gamma(n/2)}{\Gamma(n)}\right)^{\frac{2\gamma}{n}} . \]
This follows easily from~\cite[Corollary~5.3]{GonzalezQing2010} and conformal covariance.

The key observation in the proof of Theorem~\ref{thm:sobolev} is that the right-hand side of~\eqref{eqn:sobolev} is the energy of the weighted Paneitz operator on $(S_+^{n+1},d\theta^2,x_{n+1},m,1)$.  The relation to the $L^{\frac{2n}{n-2\gamma}}$-norm of the trace then follows from Theorem~\ref{thm:intro_inequality2} and the sharp fractional Sobolev inequality~\cite{Beckner1993,CotsiolisTavoularis2004,FrankLieb2012b,Lieb1983}.  In fact, Theorem~\ref{thm:sobolev} can be extended to a much more general class of functions $U$ and a large class of conformally flat metrics on the upper hemisphere; see Theorem~\ref{thm:sobolev_genl}.

This article is organized as follows:

In Section~\ref{sec:bg} we recall some facts about both the fractional GJMS operators as defined via scattering theory~\cite{GrahamZworski2003} and smooth metric measure spaces as used to study fractional GJMS operators via extensions~\cite{CaseChang2013}.

In Section~\ref{sec:boundary} we introduce conformally covariant boundary operators which, when coupled with the weighted conformal Laplacian and weighted Paneitz operator, are formally self-adjoint.

In Section~\ref{sec:asymptotics} we give formulae for our conformally covariant operators in terms of the asymptotics of compactifications of Poincar\'e--Einstein manifolds and thereby obtain new interpretations of the fractional GJMS operators via extensions.

In Section~\ref{sec:inequality} we give characterizations for when the left-hand sides of~\eqref{eqn:intro_inequality} and~\eqref{eqn:intro_inequality2} are uniformly bounded below and also state and prove more refined versions of Theorem~\ref{thm:intro_inequality} and Theorem~\ref{thm:intro_inequality2}.

In Section~\ref{sec:sobolev} we prove the more general version of Theorem~\ref{thm:sobolev}.

In Appendix~\ref{sec:appendix} we prove a family of Sobolev trace theorems which are relevant to this article and slightly different from the usual ones.

\subsection*{Acknowledgments}

I would like to thank Antonio Ache and Alice Chang for discussions relating to their article~\cite{AcheChang2015} which helped shape the investigations of Section~\ref{sec:asymptotics}.  I would also like to thank Rod Gover for pointing out the reference~\cite{Grant2003}, Mar\'ia del Mar Gonz\'alez for comments on an early draft of this article, and the referees for helpful comments which helped improve the clarity of the exposition.

\section{Background}
\label{sec:bg}

\subsection{Scattering theory}
\label{subsec:bg/scattering}

A \emph{Poincar\'e--Einstein manifold} is a triple $(X^{n+1},M^n,g_+)$ consisting of a complete Einstein manifold $(X^{n+1},g_+)$ with $\Ric(g_+)=-ng_+$ and $n\geq3$ such that $X$ is diffeomorphic to the interior of a compact manifold $\oX$ with boundary $M=\partial\oX$.  We further require the existence of a defining function for $M$; i.e.\ a smooth nonnegative function $\rho\colon\oX\to\bR$ such that $\rho^{-1}(0)=M$, the metric $g:=\rho^2g_+$ extends to a $C^{n-1,\alpha}$ metric on $\oX$, and $\lv d\rho\rv_g^2=1$ on $M$.  If $\rho$ is a defining function for $M$, then so too is $e^\sigma\rho$ for any $\sigma\in C^\infty(\oX)$, and hence only the conformal class $[g\rv_{TM}]$ on $M$ is well-defined.  An element $h\in[g\rv_{TM}]$ is a \emph{representative of the conformal boundary}, and to each such representative there is a defining function $r$, unique in a neighborhood of $M$ and called the \emph{geodesic defining function}, such that $g_+=r^{-2}\left(dr^2 + h_r\right)$ near $M$ for $h_r$ a one-parameter family of Riemannian metrics on $M$ with
\begin{align*}
 h_r & = h + h_{(2)}r^2 + \dotso + h_{(n-1)}r^{n-1} + kr^n + o(r^n), \quad\text{if $n$ is odd}, \\
 h_r & = h + h_{(2)}r^2 + \dotso + h_{(n-2)}r^{n-2} + h_{(n)}r^n\log r + kr^n + o(r^n), \quad\text{if $n$ is odd},
\end{align*}
where the terms $h_{(\ell)}$ for $\ell\leq n$ even are locally determined by $h$ while the term $k$ is nonlocal.  For example, $h_{(2)}=-\frac{1}{n-2}\left(\Ric_h-\frac{1}{2(n-1)}R_h\,h\right)$ is the negative of the \emph{Schouten tensor} of $h$.  For further details, including a discussion of optimal regularity, see~\cite{ChruscielDelayLeeSkinner2005} and the references therein.

Given a Poincar\'e--Einstein manifold $(X^{n+1},M^n,g_+)$, a representative $h$ of the conformal boundary, and a parameter $\gamma\in(0,\frac{n}{2})\setminus\bN$ such that $\frac{n^2}{4}-\gamma^2$ does not lie in the $L^2$-spectrum of $-\Delta_{g_+}$, we define the \emph{fractional GJMS operator} $P_{2\gamma}$ as follows:  Let $s=\frac{n}{2}+\gamma$.  For any $f\in C^\infty(M)$, there exists a unique solution $v$, denoted $\mP(\frac{n}{2}+\gamma)f$, of the generalized eigenvalue problem
\begin{subequations}
\label{eqn:poisson}
\begin{equation}
\label{eqn:poisson_equation}
-\Delta_{g_+}v - s(n-s)v = 0
\end{equation}
such that, asymptotically near $M$,
\begin{equation}
\label{eqn:poisson_expansion}
v = Fr^{n-s} + Gr^s
\end{equation}
\end{subequations}
for $F,G\in C^\infty(\oX)$ and $F\rv_M=f$.  Then
\begin{equation}
\label{eqn:fractional_gjms}
P_{2\gamma}f := d_\gamma G\rv_M \quad\text{for}\quad d_\gamma = 2^{2\gamma}\frac{\Gamma(\gamma)}{\Gamma(-\gamma)} .
\end{equation}
Among the key properties of the fractional GJMS operator $P_{2\gamma}\colon C^\infty(M)\to C^\infty(M)$ are that it is formally self-adjoint, that its principal symbol is that of $(-\Delta)^\gamma$, and that it is conformally covariant; indeed, if $\hat h=e^{2\sigma}h$ is another representative of the conformal boundary, then
\[ \hP_{2\gamma}\left(f\right) = e^{-\frac{n+2\gamma}{2}\sigma}P_{2\gamma}\left(e^{\frac{n-2\gamma}{2}\sigma}f\right) \]
for all $f\in C^\infty(M)$.  In fact, this definition extends to the cases $\gamma\in\bN$ by analytic continuation, and in these cases the operators $P_{2\gamma}$ recover the GJMS operators.  For further details, see~\cite{GrahamZworski2003}.

A useful fact about the solution $v$ of~\eqref{eqn:poisson} is that, up to order $r^{n/2}$, the Taylor series expansion of $F$ (resp.\ $G$) is even in $r$ and depends only on $h$ and $F\rv_M$ (resp.\ $G\rv_M$).  For example,
\begin{equation}
 \label{eqn:Fexpansion2}
 F = f + \frac{1}{4(1-\gamma)}\left(-\oDelta f + \frac{n-2\gamma}{2}\oJ f\right)r^2 + o(r^2),
\end{equation}
where $\oJ$ is the trace (with respect to $h$) of the Schouten tensor $\oP$ and we adopt the convention that barred operators are defined with respect to the boundary $(M^n,h)$.

The fractional GJMS operators $P_{2\gamma}$ can be interpreted as generalized Dirichlet-to-Neumann operators associated to weighted GJMS operators.  To state this precisely and in the widest generality in which we are interested requires a discussion of smooth metric measure spaces.

\subsection{Smooth metric measure spaces}
\label{subsec:bg/smms}

A \emph{smooth metric measure space} is a five-tuple $(\oX^{n+1},g,\rho,m,1)$ formed from a smooth manifold $\oX^{n+1}$ with (possibly empty) boundary $M^n=\partial\oX$, a Riemannian metric $g$ on $\oX$, a nonnegative function $\rho\in C^\infty(\oX)$ with $\rho^{-1}(0)=M$, and a dimensional constant $m\in(1-n,\infty)$.  Given such a smooth metric measure space, we always denote by $X$ the interior of $\oX$.  Heuristically, the interior of a smooth metric measure space represents the base of a warped product
\begin{equation}
\label{eqn:warped_product}
\left( X^{n+1}\times S^m, g\oplus\rho^2 d\theta^2\right)
\end{equation}
for $(S^m,d\theta^2)$ the $m$-sphere with a metric of constant sectional curvature one; this is the meaning of the $1$ as the fifth element of the five-tuple defining a smooth metric measure space.  The choice of the standard $m$-sphere allows us to partially compactify~\eqref{eqn:warped_product}, though not necessarily smoothly, by adding the boundary $M$ of $X$.  The model case is the upper half space $(\bR_+\times\bR^n,dy^2\oplus dx^2,y,m,1)$ for $y$ the coordinate on $\bR_+:=(0,\infty)$; in this case the warped product~\eqref{eqn:warped_product} is the flat metric on $\bR^{n+m+1}\setminus\{0\}$, and the partial compactification obtained from $[0,\infty)\times\bR^n$ is the whole of $\bR^{n+m+1}$.

The heuristic of passing through the warped product~\eqref{eqn:warped_product} is useful in that most geometric invariants defined on a smooth metric measure space --- and all which are considered in this article --- can be formally obtained by considering their Riemannian counterparts on~\eqref{eqn:warped_product} while restricting to the base $X$.  More precisely, when $m\in\bN$, the warped product~\eqref{eqn:warped_product} makes sense and one can define invariants on $X$ in terms of Riemannian invariants on~\eqref{eqn:warped_product} by means of the canonical projection $\pi\colon X^{n+1}\times S^m\to X^{n+1}$.  Invariants obtained in this way are polynomial in $m$, and can be extended to general $m\in(1-n,\infty)$ by treating $m$ as a formal variable.  This is illustrated by means of specific examples below.

The \emph{weighted Laplacian} $\Delta_\phi\colon C^\infty(X)\to C^\infty(X)$ is defined by
\[ \Delta_\phi U := \Delta U + m\rho^{-1}\lp\nabla\rho,\nabla U\rp . \]
This operator is formally self-adjoint with respect to the measure $\rho^m\dvol_g$; the notation $\Delta_\phi$ is used for consistency with the literature on smooth metric measure spaces, where one usually writes $\rho^m=e^{-\phi}$ and allows $m$ to become infinite.  In terms of~\eqref{eqn:warped_product}, one readily checks that $\pi^\ast\Delta_\phi U = \boldsymbol{\Delta}(\pi^\ast U)$ for $\boldsymbol{\Delta}$ the Laplacian of~\eqref{eqn:warped_product}.  The \emph{weighted Schouten scalar} $J_\phi^m$ and the \emph{weighted Schouten tensor} $P_\phi^m$ are the tensors
\begin{align*}
J_\phi^m & := \frac{1}{2(m+n)}\left(R - 2m\rho^{-1}\Delta\rho - m(m-1)\rho^{-2}\left(\lv\nabla\rho\rv^2-1\right)\right), \\
P_\phi^m & := \frac{1}{m+n-1}\left(\Ric - m\rho^{-1}\nabla^2\rho - J_\phi^m\right).
\end{align*}
Denoting by $\boldsymbol{P}$ the Schouten tensor of~\eqref{eqn:warped_product} and by $\boldsymbol{J}$ its trace, one readily checks that $\boldsymbol{J}=\pi^\ast J_\phi^m$ and that $P_\phi^m(Z,Z)=\boldsymbol{P}(\tilde Z,\tilde Z)$ for all $Z\in TX$, where $\tilde Z$ is the horizontal lift of $Z$ to $X\times S^m$.  The \emph{weighted conformal Laplacian} $L_{2,\phi}^m\colon C^\infty(X)\to C^\infty(X)$ and the \emph{weighted Paneitz operator} $L_{4,\phi}^m\colon C^\infty(X)\to C^\infty(X)$ are defined by
\begin{align*}
L_{2,\phi}^mU & := -\Delta_\phi U + \frac{m+n-1}{2}J_\phi^m U , \\
L_{4,\phi}^mU & := (-\Delta_\phi)^2U + \delta_\phi\left((4P_\phi^m-(m+n-1)J_\phi^mg)(\nabla U)\right) + \frac{m+n-3}{2}Q_\phi^mU
\end{align*}
where $\delta_\phi X=\tr_g\nabla X + m\rho^{-1}\lp X,\nabla\rho\rp$ is the negative of the formal adjoint of the gradient with respect to $\rho^m\dvol$,
\[ Q_\phi^m := -\Delta_\phi J_\phi^m - 2\lv P_\phi^m\rv^2 - \frac{2}{m}\left(Y_\phi^m\right)^2 + \frac{m+n-1}{2}\left(J_\phi^m\right)^2 \]
is the \emph{weighted $Q$-curvature}, and $Y_\phi^m=J_\phi^m-\tr_g P_\phi^m$.  Observe that the weighted conformal Laplacian and the weighted Paneitz operator are both formally self-adjoint with respect to $\rho^m\dvol$.  These definitions recover the conformal Laplacian and the Paneitz operator, respectively, of~\eqref{eqn:warped_product} when restricted to the base.

An important property of the weighted conformal Laplacian and the weighted Paneitz operator is that they are both conformally covariant.  Two smooth metric measure spaces $(\oX^{n+1},g,\rho,m,1)$ and $(\oX^{n+1},\hg,\hrho,m,1)$ are \emph{pointwise conformally equivalent} if there is a function $\sigma\in C^\infty(\oX)$ such that $\hg=e^{2\sigma}g$ and $\hrho=e^\sigma\rho$.  This is equivalent to requiring that the respective warped products~\eqref{eqn:warped_product} are pointwise conformally equivalent with conformal factor independent of $S^m$.  Under this assumption, it holds that
\begin{align}
\label{eqn:conformally_covariant2} \widehat{L_{2,\phi}^m}(U) & = e^{-\frac{m+n+3}{2}\sigma}L_{2,\phi}^m\left(e^{\frac{m+n-1}{2}\sigma}U\right), \\
\label{eqn:conformally_covariant4} \widehat{L_{4,\phi}^m}(U) & = e^{-\frac{m+n+5}{2}\sigma}L_{2,\phi}^m\left(e^{\frac{m+n-3}{2}\sigma}U\right)
\end{align}
for all $U\in C^\infty(X)$.

As defined above, the weighted conformal Laplacian and the weighted Paneitz operator are defined only in the interior of a smooth metric measure space.  The purpose of this article is to introduce and study boundary operators associated to the weighted conformal Laplacian and the Paneitz operator, respectively, which share their conformal covariance and formal self-adjointness properties.  To do this in such a way as to meaningfully study Poincar\'e--Einstein manifolds and the fractional GJMS operators requires us to allow weaker-than-$C^\infty$ regularity for both the metric $g$ and the function $\rho$ at the boundary of our smooth metric measure spaces.  This requires some definitions.

\begin{defn}
 \label{defn:geodesic}
 Let $(\oX^{n+1},g)$ be a Riemannian manifold with nonempty boundary $M=\partial\oX$.  Let $\gamma\in(0,n/2)\setminus\bN$ and set $k=\lfloor\gamma\rfloor$ and $m=1+2k-2\gamma$.  The smooth metric measure space $(\oX,g,r,m,1)$ is \emph{geodesic} if $\lv\nabla r\rv^2=1$ in a neighborhood of $M$ and if
 \begin{equation}
  \label{eqn:geodesic_metric}
  g = dr^2 + \sum_{j=0}^{k} h_{(2j)}r^{2j} + o(r^{2\gamma})
 \end{equation}
 for sections $h_{(0)},\dotsc,h_{(2k)}$ of $S^2T^\ast M$.
\end{defn}

The asymptotic expansion~\eqref{eqn:geodesic_metric} is to be understood in the following way: Each point $p\in M$ admits an open neighborhood $U\subset\oX$ and a constant $\varepsilon>0$ such that the map
\begin{equation}
 \label{eqn:geodesic_diffeomorphism}
 [0,\varepsilon)\times V \ni (t,q) \mapsto \gamma_q(t) \in U
\end{equation}
is a diffeomorphism with image $U$, where $V:=U\cap M$ and $\gamma_q$ is the integral curve in the direction $\nabla r$ originating at $q$.  By shrinking $U$ if necessary, we may assume that $\lv\nabla r\rv^2=1$ in $U$, and hence $r(\gamma_q(t))=t$; note that if $M$ is compact, then we may take $U$ to be a neighborhood of $M$.  The composition of the canonical projection $[0,\varepsilon)\times V\to V$ with the inverse of the diffeomorphism~\eqref{eqn:geodesic_diffeomorphism} gives a map $\pi\colon U\to V$.  We then consider covariant tensor fields on $V$ as covariant tensor fields in $U$ by pulling them back by $\pi$.  Finally, since $\lv\nabla r\rv^2=1$ in a neighorhood of $M$, it is straightforward to check that there is a one-parameter family $h_r$ of sections of $S^2T^\ast M$ such that $g=dr^2+h_r$ near $M$.  The assumption~\eqref{eqn:geodesic_metric} imposes the additional requirement that $h_r$ is even in $r$ to order $o(r^{2\gamma})$.  In particular, if $\gamma>1/2$, then $M$ is totally geodesic with respect to $g$; if also $\gamma>3/2$, then the scalar curvature $R$ of $g$ satisfies $\partial_rR=0$ along $M$.

Note that if $r$ is a geodesic defining function for a Poincar\'e--Einstein manifold $(X^{n+1},M^n,g_+)$ and if $m,\gamma$ are as in Definition~\ref{defn:geodesic}, then $(\oX,r^2g_+,r,m,1)$ is a geodesic smooth metric measure space.

\begin{defn}
 \label{defn:gamma_admissible}
 Let $\oX^{n+1}$ be a smooth manifold with boundary $M=\partial\oX$ and let $\gamma\in(0,n/2)\setminus\bN$.  Set $k=\lfloor\gamma\rfloor$ and $m=1+2k-2\gamma$.  A smooth metric measure space $(\oX^{n+1},g,\rho,m,1)$ is \emph{$\gamma$-admissible} if it is pointwise conformally equivalent to a geodesic smooth metric measure space $(\oX,g_0,r,m,1)$ such that
 \begin{equation}
  \label{eqn:gamma_admissible}
  \frac{\rho}{r} = \sum_{j=0}^k \rho_{(2j)}r^{2j} + \Phi r^{2\gamma} + o(r^{2\gamma})
 \end{equation}
 for $\rho_{(0)},\dotsc,\rho_{(2k)},\Phi\in C^\infty(M)$ and $\rho_{(0)}=1$.
\end{defn}

Note that if $(\oX,g,\rho,m,1)$ is a $\gamma$-admissible smooth metric measure space and there are two geodesic smooth metric measure spaces $(\oX,g_i,r_i,m,1)$, $i\in\{1,2\}$, as in Definition~\ref{defn:gamma_admissible}, then $r_2=r_1$ near $M$ (cf.\ \cite[Lemma~5.2]{GrahamLee1991} or~\cite[Lemma~5.1]{Lee1995}); in particular, all asymptotic statements about $\gamma$-admissible smooth metric measure spaces (e.g.\ \eqref{eqn:gamma_admissible}) are independent of the choice of geodesic smooth metric measure space in Definition~\ref{defn:gamma_admissible}.  Combining the expansions~\eqref{eqn:geodesic_metric} and~\eqref{eqn:gamma_admissible}, we see that if $(\oX,g,\rho,m,1)$ is a $\gamma$-admissible smooth metric measure space with $\gamma>1/2$, then $M$ is totally geodesic (with respect to $g$); if also $\gamma>3/2$, then $\partial_\rho R=0$ along $M$.

Given a Poincar\'e--Einstein manifold $(X^{n+1},M^n,g_+)$ and $\gamma\in(0,n/2)\setminus\bN$, a defining function $\rho$ is \emph{$\gamma$-admissible} if $(\oX,\rho^2g_+,\rho,m,1)$, $m=1-2\lfloor\gamma\rfloor-2\gamma$, is a $\gamma$-admissible smooth metric measure space.  In particular, the extension theorems established by Chang and the author~\cite[Theorem~4.1 and Theorem~4.4]{CaseChang2013} are all stated in terms of $\gamma$-admissible smooth metric measure spaces.  An important example of $\gamma$-admissible smooth metric measure spaces which arise as compactifications of Poincar\'e--Einstein manifolds and for which the function $\Phi$ in~\eqref{eqn:gamma_admissible} is not necessarily zero are obtained from the adapted defining function~\cite[Subsection~6.1]{CaseChang2013}.

In light of both our weakened regularity hypotheses and the asymptotics of solutions to the Poisson equation~\eqref{eqn:poisson}, it is natural to introduce the following function spaces.

\begin{defn}
 Fix $\gamma\in(0,1)$, set $m=1-2\gamma$, and let $(\oX^{n+1},g,\rho,m,1)$ be a $\gamma$-admissible smooth metric measure space.  Given $f\in C^\infty(M)$, denote by $\mC_f^\gamma$ the set of all $U\in C^\infty(X)\cap C^0(\oX)$ such that, asymptotically near $M$,
 \begin{equation}
  \label{eqn:mC0f_asympt}
  U = f + \psi\rho^{2\gamma} + o(\rho^{2\gamma})
 \end{equation}
 for some $\psi\in C^\infty(M)$.  Set
 \begin{equation}
  \label{eqn:mC0}
  \mC^\gamma := \bigcup_{f\in C^\infty(M)} \mC_f^\gamma .
 \end{equation}
 The Sobolev spaces $W_0^{1,2}(\oX,\rho^m\dvol)$ and $W^{1,2}(\oX,\rho^m\dvol)$ are the completions of $\mC_0^\gamma$ and $\mC^\gamma$, respectively, with respect to the norm
 \[ \left\lV U\right\rV_{W^{1,2}}^2 := \int_X \left( \lv\nabla U\rv^2 + U^2\right)\rho^m\dvol . \]
\end{defn}

For notational convenience, in the case $\gamma\in(0,1)$ we sometimes denote by $\mD^\gamma$ the space $\mC^\gamma$ and by $\mH^\gamma$ the space $W^{1,2}(\oX,\rho^m\dvol)$.

When $\gamma\in(0,1)$, the Sobolev trace theorem (e.g.\ \cite{Triebel1978}) states that there is a surjective bounded linear operator $\Tr\colon\ W^{1,2}(\oX,\rho^m\dvol)\to H^\gamma(M)$ such that $\Tr U = f$ for every $U\in\mC_f^\gamma$, where $H^\gamma(M)$ denotes the completion of $C^\infty(M)$ with respect to the norm obtained by pulling back
\[ \lV f\rV_{H^\gamma(\bR^n)}^2 := \int_{\bR^n} f^2dx + \int_{\bR^n}\int_{\bR^n} \frac{\lv f(x) - f(y)\rv^2}{\lv x-y\rv^{n+2\gamma}} dx\,dy \]
to $M$ via coordinate charts.

\begin{defn}
 Fix $\gamma\in(1,2)$, set $m=3-2\gamma$, and let $(\oX^{n+1},g,\rho,m,1)$ be a $\gamma$-admissible smooth metric measure space.  Given $f,\psi\in C^\infty(M)$, denote by $\mC_{f,\psi}^\gamma$ the set of all $U\in C^\infty(X)\cap C^0(\oX)$ such that, asymptotically near $M$,
 \begin{equation}
  \label{eqn:mC1f_asympt}
  U = f + \psi\rho^{2\gamma-2} + f_2\rho^2 + \psi_2\rho^{2\gamma} + o(\rho^{2\gamma})
 \end{equation}
 for some $f_2,\psi_2\in C^\infty(M)$.  Set
 \begin{align}
  \label{eqn:mC1} \mC^\gamma & := \bigcup_{f,\psi\in C^\infty(M)} \mC_{f,\psi}^\gamma , \\
  \label{eqn:mD1} \mD^\gamma & := \bigcup_{f\in C^\infty(M)} \mC_{f,0}^\gamma .
 \end{align}
 The Sobolev spaces $W_0^{2,2}(\oX,\rho^m\dvol)$, $W^{2,2}(\oX,\rho^m\dvol)$, and $\mH^\gamma$ are the completions of $\mC_{0,0}^\gamma$, $\mC^\gamma$, and $\mD^\gamma$, respectively, with respect to the norm
 \begin{equation}
  \label{eqn:w22}
  \left\lV U\right\rV_{W^{2,2}}^2 := \int_X \left( \left|\nabla^2U+m\rho^{-1}(\partial_\rho U)^2d\rho\otimes d\rho\right|^2 + \lv\nabla U\rv^2 + U^2\right) \rho^m\dvol .
 \end{equation}
\end{defn}

The particular modification of the Hessian used in~\eqref{eqn:w22} ensures that the integral is finite for all $U\in\mC^\gamma$.  Given $U\in\mC_{f,\psi}^\gamma$, the weighted Bochner formula (cf.\ Appendix~\ref{sec:appendix}) allows one to rewrite this Hessian term in terms of the $L^2$-norm of $\Delta_\phi U$, lower order interior terms depending on curvature, and boundary terms involving only $f$ and $\psi$. 

When $\gamma\in(1,2)$, the Sobolev trace theorem (see Appendix~\ref{sec:appendix}) states that there is a surjective bounded linear operator $\Tr\colon W^{2,2}(\oX,\rho^m\dvol)\to H^\gamma(M)\oplus H^{2-\gamma}(M)$ such that $\Tr(U)=(f,\psi)$ for every $U\in\mC_{f,\psi}^\gamma$, where $H^\gamma(M)$ denotes the completion of $C^\infty(M)$ with respect to the norm obtained by pulling back
\[ \lV f\rV_{H^\gamma(\bR^n)}^2 := \int_{\bR^n} \left( f^2 + \lv\nabla f\rv^2\right)dx + \sum_{j=1}^n\int_{\bR^n}\int_{\bR^n} \frac{\lv \partial_jf(x)-\partial_jf(y)\rv^2}{\lv x-y\rv^{n+2\gamma-2}} dx\,dy \]
to $M$ via coordinate charts.

We conclude with two useful observations.  The first is the following relationship between a defining function for a Poincar\'e--Einstein manifold and certain weighted geometric invariants of the induced compactification.

\begin{lem}[{\cite[Lemma~3.2]{CaseChang2013}}]
 \label{lem:pe_smms_formulae}
 Let $(X^{n+1},M^n,g_+)$ be a Poincar\'e--Einstein manifold and let $\rho$ be a defining function.  Fix $m>1-n$.  The smooth metric measure space $(X^{n+1},g:=\rho^2g_+,\rho,m,1)$ has
 \begin{align}
  \label{eqn:J_to_grad} J + \rho^{-1}\Delta\rho & = \frac{n+1}{2}\rho^{-2}\left(\lv\nabla\rho\rv^2 - 1 \right) \\
  \label{eqn:J_to_weight} J_\phi^m & = J - \frac{m}{n+1}\left(J + \rho^{-1}\Delta\rho\right) \\
  \label{eqn:P_to_weight} P_\phi^m & = P .
 \end{align}
\end{lem}

Here $P$ and $J$ are the Schouten tensor of $g$ and its trace, respectively.

The second is the following characterization of pointwise conformally equivalent $\gamma$-admissible smooth metric measure spaces in terms of the conformal factors.

\begin{lem}
 \label{lem:gamma_admissible_mC}
 Fix $\gamma\in(0,2)\setminus\{1\}$ and let $(\oX^{n+1},g,\rho,m,1)$ be a $\gamma$-admissible smooth metric measure space with $m=1+2\lfloor\gamma\rfloor-2\gamma$.  Let $\sigma\in C^\infty(X)\cap C^0(\oX)$ and set $\hg=e^{2\sigma}g$ and $\hrho=e^\sigma\rho$.  Then $(\oX^{n+1},\hg,\hrho,m,1)$ is a $\gamma$-admissible smooth metric measure space if and only if $\sigma\in\mD^\gamma$.
\end{lem}

\begin{proof}
 Let $(\oX,g_0,r,m,1)$ and $(\oX,\hg_0,\hr,m,1)$ be geodesic smooth metric measure spaces associated to $(\oX,g,\rho,m,1)$ and $(\oX,\hg,\hr,m,1)$, respectively, as in Definition~\ref{defn:gamma_admissible}.  Suppose first that $(\oX,\hg,\hrho,m,1)$ is $\gamma$-admissible.  We readily check that
 \[ e^\sigma = \frac{\hrho}{\hr}\cdot\frac{\hr}{r}\cdot\frac{r}{\rho} \in \mD^\gamma, \]
 whence $\sigma\in\mD^\gamma$.  Conversely, if $\sigma\in\mD^\gamma$, we readily check that $\hr/r\in\mD^\gamma$, whence $(\oX^{n+1},\hg,\hrho,m,1)$ is $\gamma$-admissible.
\end{proof}

For the remainder of this article, unless otherwise specified, the measure with respect to which an integral is evaluated is specified by context: If a smooth metric measure space $(\oX^{n+1},g,\rho,m,1)$ with boundary $M=\partial\oX$ is given, all integrals over $X$ are evaluated with respect to $\rho^m\dvol_g$ and all integrals over $M$ are evaluated with respect to the Riemannian volume element of $g\rv_{TM}$.
\section{The conformally covariant boundary operators}
\label{sec:boundary}

In order to study boundary value problems associated to the weighted conformal Laplacian and the weighted Paneitz operator --- for instance, to study the fractional GJMS operators as in~\cite{CaseChang2013} --- it is useful to find conformally covariant boundary operators associated to these respective operators.  In the case of the weighted conformal Laplacian $L_{2,\phi}^m$ with $m=1-2\gamma$, this means finding conformally covariant operators $B_0^{2\gamma}$ and $B_{2\gamma}^{2\gamma}$ such that $(L_{2,\phi}^mU,V)=(L_{2,\phi}^mV,U)$ for all $U,V\in\ker B_0^{2\gamma}$ or for all $U,V\in\ker B_{2\gamma}^{2\gamma}$.  That is, the boundary value problems $(L_{2,\phi}^m;B_0^{2\gamma})$ and $(L_{2,\phi}^m;B_{2\gamma}^{2\gamma})$ are formally self-adjoint.  In the case of the weighted Paneitz operator $L_{4,\phi}^m$ with $m=3-2\gamma$, this means defining conformally covariant operators $B_0^{2\gamma}, B_{2\gamma-2}^{2\gamma}, B_2^{2\gamma}, B_{2\gamma}^{2\gamma}$ such that $(L_{4,\phi}^mU,V)=(L_{4,\phi}^mV,U)$ for all $U,V$ in the kernel of one of the pairs $\kB_1=(B_0^{2\gamma},B_2^{2\gamma})$, $\kB_2=(B_0^{2\gamma},B_{2\gamma-2}^{2\gamma})$, or $\kB_3=(B_{2\gamma-2}^{2\gamma},B_{2\gamma}^{2\gamma})$.  That is, the boundary value problems $(L_{4,\phi}^m;\kB_j)$ for $j\in\{1,2,3\}$ are all formally self-adjoint.  These boundary value problems are all elliptic, as is apparent from the definitions of the operators given below, and our definitions are such that the formal self-adjointness follows from simple integration-by-parts identities; see Theorem~\ref{thm:robin} for the case of the weighted conformal Laplacian and Theorem~\ref{thm:invariant} and Theorem~\ref{thm:integral} for the case of the weighted Paneitz operator.

The existence of such operators when $m=0$ is already known: $B=\eta+\frac{n-1}{2n}H$ is a boundary operator for the conformal Laplacian (cf.\ \cite{Branson1997,Escobar1990}), while Branson and Gover~\cite{BransonGover2001} have constructed via the tractor calculus conformally covariant boundary operators associated to the non-critical GJMS operators and Grant~\cite{Grant2003} derived the third-order boundary operator associated to the Paneitz operator (see also~\cite{ChangQing1997a,Juhl2009} for the case of critical dimension).  As is apparent from Definition~\ref{defn:operators0}, $B_1^1=B$, while the operators $B_k^3$ for $k\in\{0,1,2,3\}$ give explicit formulae for the boundary operators associated to the Paneitz operator in the case of manifolds with totally geodesic boundary; see~\cite{Case2015b} for the general case.

\subsection{The case $\gamma\in(0,1)$}
\label{subsec:boundary/1}

The conformally covariant boundary operators associated to the weighted conformal Laplacian are defined as follows.

\begin{defn}
 \label{defn:operators0}
 Fix $\gamma\in(0,1)$ and set $m=1-2\gamma$.  Let $(\oX^{n+1},g,\rho,m,1)$ be a $\gamma$-admissible smooth metric measure space with boundary $M=\partial\oX$ and let $(\oX,g_0,r,m,1)$ be the geodesic smooth metric measure space as in Definition~\ref{defn:gamma_admissible}.  Set $\eta=-\frac{\rho}{r}\nabla^gr$.  As operators mapping $\mC^\gamma$ to $C^\infty(M)$,
 \begin{align*}
  B_0^{2\gamma}U & := U\rv_M , \\
  B_{2\gamma}^{2\gamma}U & := \lim_{\rho\to0} \rho^m\left(\eta U + \frac{n-2\gamma}{2n}U\delta\eta\right) .
 \end{align*}
 where $\delta\eta := \tr_g \nabla^g\eta$.
\end{defn}

Note that $\eta$ is the outward-pointing unit normal (with respect to $g$) vector field along the level sets of $r$ in a neighorhood of $M$.  In particular, if $\gamma=1/2$, then $\delta\eta\rv_M=H$ is the mean curvature of $M$ with respect to $g$.  For this reason, we call
\[ H_{2\gamma} := \lim_{\rho\to0} \rho^m\delta\eta \]
the \emph{$\gamma$-mean curvature} of $M$.  Since $(\oX^{n+1},g_0,r,m,1)$ is uniquely determined near $M$ by $(\oX,g,\rho,m,1)$, the asymptotic assumptions of Definition~\ref{defn:operators0} guarantee that the $\gamma$-mean curvature and the operators $B_{0}^{2\gamma}$ and $B_{2\gamma}^{2\gamma}$ are well-defined; indeed,
\begin{align*}
 H_{2\gamma} & = -2n\gamma\Phi , \\
 B_0^{2\gamma}U & = f, \\
 B_{2\gamma}^{2\gamma}U & = -2\gamma\left(\psi + \frac{n-2\gamma}{2}\Phi f\right) ,
\end{align*}
where $\rho$ and $U$ satisfy~\eqref{eqn:gamma_admissible} and~\eqref{eqn:mC0f_asympt}, respectively, near $M$.

That the operators $B_{0}^{2\gamma}$ and $B_{2\gamma}^{2\gamma}$ are the conformally covariant boundary operators associated to the weighted conformal Laplacian is a consequence of the following result.

\begin{thm}
 \label{thm:robin}
 Fix $\gamma\in(0,1)$ and set $m=1-2\gamma$.  Let $(\oX^{n+1},g,\rho,m,1)$ and $(\oX^{n+1},\hg,\hrho,m,1)$ be two pointwise conformally equivalent $\gamma$-admissible smooth metric measure spaces with $\hg=e^{2\sigma}g$ and $\hrho=e^\sigma\rho$.  Then for any $U\in\mC^\gamma$ it holds that
 \begin{align}
  \label{eqn:B0_0_covariant} \hB_{0}^{2\gamma}(U) & = e^{-\frac{n-2\gamma}{2}\sigma\rv_M}B_{0}^{2\gamma}\left(e^{\frac{n-2\gamma}{2}\sigma} U\right), \\
  \label{eqn:B1_0_covariant} \hB_{2\gamma}^{2\gamma}\left(U\right) & = e^{-\frac{n+2\gamma}{2}\sigma\rv_M} B_{2\gamma}^{2\gamma}\left( e^{\frac{n-2\gamma}{2}\sigma}U\right) .
 \end{align}
 Moreover, given $U,V\in\mC^\gamma$, it holds that
 \begin{equation}
  \label{eqn:robin_ibp}
  \int_X V\,L_{2,\phi}^mU + \oint_M B_{0}^{2\gamma}(V)B_{2\gamma}^{2\gamma}(U) = \mQ_{2\gamma}(U,V)
 \end{equation}
 for $\mQ_{2\gamma}$ the symmetric bilinear form
 \[ \mQ_{2\gamma}(U,V) = \int_X \left(\lp\nabla U,\nabla V\rp + \frac{n-2\gamma}{2}J_\phi^mUV\right) + \frac{n-2\gamma}{2n}\oint_M H_{2\gamma}B_0^{2\gamma}(U)B_0^{2\gamma}(V) . \]
 In particular, $\mQ_{2\gamma}$ is conformally covariant.
\end{thm}

\begin{proof}
 \eqref{eqn:B0_0_covariant} follows immediately from the definition of $B_{0}^{2\gamma}$.

 By Lemma~\ref{lem:gamma_admissible_mC}, we have that $\sigma\in\mC^\gamma$, and in particular $\rho^m\eta\sigma$ is well-defined.  On the other hand, if $\eta$ and $\heta$ are as in Definition~\ref{defn:operators0}, then $\heta=e^{-\sigma}\eta$.  Hence
 \begin{align*}
  \hrho\,{}^m\heta\left(e^{-\frac{n-2\gamma}{2}\sigma}U\right) & = e^{-\frac{n+2\gamma}{2}\sigma}\rho^m\left(\eta U - \frac{n-2\gamma}{2}U\eta\sigma\right), \\
  \hrho\,{}^m\hdelta\heta & = e^{-2\gamma\sigma}\rho^m\left(\delta\eta + n\eta\sigma\right) .
 \end{align*}
 Combining these two equations yields~\eqref{eqn:B1_0_covariant}.

 Finally, integration by parts yields~\eqref{eqn:robin_ibp}.  Combining~\eqref{eqn:B0_0_covariant} and~\eqref{eqn:B1_0_covariant} with~\eqref{eqn:robin_ibp} yields the conformal covariance of $\mQ_{2\gamma}$.
\end{proof}

\subsection{The case $\gamma\in(1,2)$}
\label{subsec:boundary/2}

The conformally covariant boundary operators associated to the weighted Paneitz operator are defined as follows.

\begin{defn}
 \label{defn:operators}
 Fix $\gamma\in(1,2)$ and set $m=3-2\gamma$.  Let $(\oX^{n+1},g,\rho,m,1)$ be a $\gamma$-admissible smooth metric measure space with boundary $M=\partial\oX$ and let $\eta$ be as in Definition~\ref{defn:operators0}.  As operators mapping $\mC^\gamma$ to $C^\infty(M)$,
 \begin{align*}
  B_{0}^{2\gamma}U & := U , \\
  B_{2\gamma-2}^{2\gamma}U & := \rho^m\eta U , \\
  B_{2}^{2\gamma}U & := -\frac{2-\gamma}{\gamma-1}\oDelta U + \left(\nabla^2U(\eta,\eta) + m\rho^{-1}\partial_\rho U\right) + \frac{n-2\gamma}{2}T_{2}^{2\gamma}U , \\
  B_{2\gamma}^{2\gamma}U & := -\rho^m\eta\Delta_\phi U - \frac{1}{\gamma-1}\oDelta\rho^m\eta U + S_2^{2\gamma}\rho^m\eta U + \frac{n-2\gamma}{2}\left(\rho^m\eta J_\phi^m\right)U ,
 \end{align*}
 where
 \begin{align}
  \label{eqn:T} T_{2}^{2\gamma} & := \frac{2-\gamma}{\gamma-1}\oJ - \left(P(\eta,\eta) - \frac{3-2\gamma}{n+1}\left(\oJ + \rho^{-1}\Delta\rho + P(\eta,\eta)\right)\right), \\
  \label{eqn:S} S_2^{2\gamma} & := \left(\frac{n-2\gamma}{2} + \frac{n+2\gamma-4}{2(\gamma-1)}\right)\oJ + \frac{n-2\gamma-4}{2}P(\eta,\eta) \\
  \notag & \quad - \frac{(3-2\gamma)(n-2\gamma+4)}{2(n+1)}\left(\oJ + \rho^{-1}\Delta\rho + P(\eta,\eta)\right)
 \end{align}
 and we understand the right-hand sides to all be evaluated in the limit $\rho\to0$.
\end{defn}

Due to the length of the computations, we break the proof that the operators given in Definition~\ref{defn:operators} are conformally covariant boundary operators associated to the weighted Paneitz operator on $\gamma$-admissible smooth metric measure spaces into two parts.  First, we show that they are conformally covariant of the correct weight.

\begin{thm}
\label{thm:invariant}
 Fix $\gamma\in(1,2)$ and set $m=3-2\gamma$.  Let $(\oX^{n+1},g,\rho,m,1)$ and $(\oX^{n+1},\hg,\hrho,m,1)$ be two pointwise conformally equivalent $\gamma$-admissible smooth metric measure spaces with $\hg=e^{2\sigma}g$ and $\hrho=e^\sigma\rho$.  Then for any $U\in\mC^\gamma$ it holds that
 \begin{align}
  \label{eqn:invariant2/0} \hB_{0}^{2\gamma}U & = e^{-\frac{n-2\gamma}{2}\sigma\rv_M}B_{0}^{2\gamma}\left(e^{\frac{n-2\gamma}{2}\sigma}U\right), \\
  \label{eqn:invariant2/1} \hB_{2\gamma-2}^{2\gamma} & = e^{-\frac{n+2\gamma-4}{2}\sigma\rv_M}B_{2\gamma-2}^{2\gamma}\left(e^{\frac{n-2\gamma}{2}\sigma}U\right), \\
  \label{eqn:invariant2/2} \hB_{2}^{2\gamma} & = e^{-\frac{n-2\gamma+4}{2}\sigma\rv_M}B_{2}^{2\gamma}\left(e^{\frac{n-2\gamma}{2}\sigma}U\right), \\
  \label{eqn:invariant2/3} \hB_{2\gamma}^{2\gamma} & = e^{-\frac{n+2\gamma}{2}\sigma\rv_M}B_{2\gamma}^{2\gamma}\left(e^{\frac{n-2\gamma}{2}\sigma}U\right) .
 \end{align}
\end{thm}

The proof of Theorem~\ref{thm:invariant} is a somewhat lengthy computation.  While such computations are routine in conformal geometry (cf.\ \cite{Branson1985,ChangQing1997a}), they have not been carried out in this form in the literature for smooth metric measure spaces, and so we sketch the details here.

Fix $\gamma\in(1,2)$.  An operator $T\colon\mC^\gamma\to C^\infty(M)$ defined on a $\gamma$-admissible smooth metric measure space $(\oX^{n+1},g,\rho,m,1)$ with boundary $M=\partial\oX$ is \emph{natural} if it can be expressed as a polynomial involving the Levi-Civita connection and the Riemann curvature tensor of $g$, powers of $\rho$, the outward-pointing normal $\eta$ along $M=\partial\oX$, and contractions thereof.  A natural operator $T$ is said to be \emph{homogeneous of degree $k\in\bR$} if for any positive constant $c\in\bR$, the operators $T$ and $\hT$ defined on $(\oX^{n+1},g,\rho,m,1)$ and $(\oX^{n+1},\hg,\hrho,m,1)$, respectively, for $\hg=c^2g$ and $\hrho=c\rho$, are related by
\[ \hT(U) = c^{k} T(U) \]
for all $U$ in the domain $\Dom(T)$ of $T$.  Given a homogeneous operator $T$ of degree $k$, a function $\sigma\in\mD^\gamma$, and a fixed weight $w\in\bR$, we denote
\begin{equation}
\label{eqn:derivative}
\left(T(U)\right)^\prime := \left.\frac{\partial}{\partial t}\right|_{t=0}\left( e^{-(w+k)t\sigma\rv_M}T_{e^{2t\sigma}g}\left(e^{wt\sigma}U\right)\right) ,
\end{equation}
where $T_{e^{2\sigma}g}$ denotes the operator $T$ as defined with respect to the smooth metric measure space $(\oX^{n+1},e^{2\sigma}g,e^\sigma\rho,m,1)$.
One readily shows (cf.\ \cite[Corollary~1.14]{Branson1985}) that, given a natural operator $T$ which is homogeneous of degree $k$ and a fixed weight $w$, it holds that
\[ T_{e^{2\sigma}g}(U) = e^{(w+k)\sigma\rv_M}T\left(e^{-w\sigma}U\right) \]
for all $\sigma\in\mD^\gamma$ and all $U\in\Dom(T)$ if and only if $\left(T(U)\right)^\prime=0$ for all $\sigma\in\mD^\gamma$ and all $U\in\Dom(T)$.

To prove Theorem~\ref{thm:invariant}, it thus suffices to compute the linearizations~\eqref{eqn:derivative} of the operators given in Definition~\ref{defn:operators} --- which are all natural and homogeneous --- with the fixed weight $w=-\frac{n-2\gamma}{2}$.  We accomplish this through a pair of lemmas.  We first consider operators which are homogeneous of degree $-2$.

\begin{lem}
 \label{lem:2}
 Fix $\gamma\in(1,2)$ and set $m=3-2\gamma$.  Let $(\oX^{n+1},g,\rho,m,1)$ be a $\gamma$-admissible smooth metric measure space with boundary $M=\partial\oX$.  Let $\sigma\in\mD^\gamma$ and let $U\in\mC^\gamma$.  Fix a weight $w\in\bR$.  Then
 \begin{align*}
  \left(\oDelta U\right)^\prime & = (n+2w-2)\lp\onabla U,\onabla\sigma\rp + wU\oDelta\sigma, \\
  \left(\nabla^2U(\eta,\eta)+m\rho^{-1}\partial_\rho U\right)^\prime & = (m+1)\lp\onabla U,\onabla\sigma\rp + wU\left(\nabla^2\sigma(\eta,\eta)+m\rho^{-1}\partial_\rho\sigma\right), \\
  \left(\oJ U\right)^\prime & = -U\oDelta\sigma , \\
  \left(UP(\eta,\eta)\right)^\prime & = -U\nabla^2\sigma(\eta,\eta), \\
  \left(\rho^{-1}U\Delta\rho\right)^\prime & = \left(\oDelta\sigma + \nabla^2\sigma(\eta,\eta) + (n+1)\rho^{-1}\partial_\rho\sigma\right)U .
 \end{align*}
\end{lem}

\begin{proof}
 Let $\hg=e^{2t\sigma}g$.  It is well-known that
 \begin{align*}
  \hP & = P - t\nabla^2\sigma + O(t^2), \\
  \hnabla^2U & = \nabla^2U - t\,dU\otimes d\sigma - t\,d\sigma\otimes dU + t\lp\nabla U,\nabla\sigma\rp_g\,g ,
 \end{align*}
 and similarly for quantities defined in terms of the induced metric on $M$.  The conclusion readily follows.
\end{proof}

We next consider operators which are homogeneous of degree $-2\gamma$.

\begin{lem}
 \label{lem:3}
 Under the same hypotheses as Lemma~\ref{lem:2}, it holds that
 \begin{align*}
  \left(\oDelta\rho^m\eta U\right)^\prime & = (2m+n+2w-4)\lp\onabla\rho^m\eta U,\onabla\sigma\rp + (m+w-1)\left(\rho^m\eta U\right)\oDelta\sigma , \\
  \left(\rho^m\eta\Delta_\phi U\right)^\prime & = (m+n+2w-1)\lp\onabla\rho^m\eta U,\onabla\sigma\rp + wU\rho^m\eta\Delta_\phi\sigma \\
  & \quad + \left((m+n+2w-1)(\nabla^2\sigma(\eta,\eta)-m\rho^{-1}\partial_\rho\sigma) + w\Delta_\phi\sigma\right)\rho^m\eta U , \\
  \left(U\rho^m\eta J_\phi^m\right)^\prime & = -U\rho^m\eta\Delta_\phi\sigma .
 \end{align*}
\end{lem}

\begin{proof}
 The first identity follows immediately from Lemma~\ref{lem:2} and the homogeneity of $\rho^m\eta$.

 The conformal transformation formula for the weighted Laplacian~\cite{Case2010a} yields
 \[ \left(\rho^m\eta\Delta_\phi U\right)^\prime = \rho^m\eta\left((m+n+2w-1)\lp\nabla U,\nabla\sigma\rp + wU\Delta_\phi\sigma\right) . \]
 A straightforward computation shows that
 \[ \rho^m\eta\left(\lp\nabla U,\nabla\sigma\rp\right) = \lp\onabla\sigma,\onabla\rho^m\eta U\rp + \left(\nabla^2\sigma(\eta,\eta) - m\rho^{-1}\partial_\rho\sigma\right)\rho^m\eta U , \]
 from which the second identity follows.

 Finally, the conformal transformation formula for the weighted scalar curvature (cf.\ \cite[Proposition~4.4]{Case2010a}) yields the last identity.
\end{proof}

\begin{proof}[Proof of Theorem~\ref{thm:invariant}]
 It is clear that~\eqref{eqn:invariant2/0} and~\eqref{eqn:invariant2/1} hold.

 It follows immediately from Lemma~\ref{lem:2} that the operator $B_{2,w}$ defined by
 \begin{align*}
  B_{2,w}U & := -\oDelta U + \frac{n+2w-2}{m+1}\left(\nabla^2U(\eta,\eta) + m\rho^{-1}\partial_\rho U\right) \\
  & \quad - w\left(\oJ - \frac{n+2w-2}{m+1}\left(P(\eta,\eta) - \frac{m}{n+1}\left(\oJ + P(\eta,\eta) + \rho^{-1}\Delta\rho\right)\right)\right)U
 \end{align*}
 satisfies $\left(B_{2,w}U\right)^\prime=0$ for any $w\in\bR$.  This yields~\eqref{eqn:invariant2/2} upon observing that $B_{2}^{2\gamma}=\frac{2-\gamma}{\gamma-1}B_{2,-\frac{n-2\gamma}{2}}$.

 It follows immediately from Lemma~\ref{lem:2} and Lemma~\ref{lem:3} that the operator $B_{2\gamma,w}$ defined by
 \begin{align*}
  B_{2\gamma,w}U & := -\rho^m\eta\Delta_\phi U + \frac{m+n+2w-1}{2m+n+2w-4}\oDelta\rho^m\eta U + T_{2,w}\rho^m\eta U - w\left(\rho^m\eta J_\phi^m\right)U, \\
  T_{2,w} & := \left(\frac{(m+w-1)(m+n+2w-1)}{2m+n+2w-4} - w\right)\oJ - (m+n+3w-1)P(\eta,\eta) \\
  & \quad - \frac{m(m+n+w-1)}{n+1}\left(\oJ + P(\eta,\eta) + \rho^{-1}\Delta\rho\right)
 \end{align*}
 satisfies $\left(B_{2\gamma,w}U\right)^\prime=0$ for any $w\in\bR$.  This yields~\eqref{eqn:invariant2/3} upon observing that $B_{2\gamma}^{2\gamma}=B_{2\gamma,-\frac{n-2\gamma}{2}}$.
\end{proof}

We next show that the operators given in Definition~\ref{defn:operators} are boundary operators associated to the weighted Paneitz operator, in the sense that the pairing
\[ \mC^\gamma\times\mC^\gamma \ni (U,V) \mapsto \left(L_{4,\phi}^mU,V\right) + \left(B_{2\gamma}^{2\gamma}U,B_{0}^{2\gamma}V\right) + \left(B_{2}^{2\gamma}U,B_{2\gamma-2}^{2\gamma}V\right) \]
is a symmetric bilinear form.  Indeed, this form can be written explicitly, and is the polarization of the energy associated to the weighted Paneitz operator on a $\gamma$-admissible smooth metric measure space with boundary.

\begin{thm}
 \label{thm:integral}
 Fix $\gamma\in(1,2)$ and set $m=3-2\gamma$.  Let $(\oX^{n+1},g,\rho,m,1)$ be a $\gamma$-admissible compact smooth metric measure space with boundary $M=\partial\oX$.  Given $U,V\in\mC^\gamma$, it holds that
 \begin{equation}
  \label{eqn:ibp}
  \int_X V\,L_{4,\phi}^mU + \oint_M \left( B_0^{2\gamma}(V)B_{2\gamma}^{2\gamma}(U) + B_{2\gamma-2}^{2\gamma}(V)B_2^{2\gamma}(U)\right) = \mQ_{2\gamma}(U,V)
 \end{equation}
 for $\mQ_{2\gamma}$ the symmetric bilinear form
 \begin{align*}
  \mQ_{2\gamma}(U,V) & = \int_X \left[ \left(\Delta_\phi U\right)\left(\Delta_\phi V\right)  - \left(4P-(n-2\gamma+2)J_\phi^mg\right)(\nabla U,\nabla V) + \frac{n-2\gamma}{2}Q_\phi^mUV \right] \\
  & \quad + \oint_M \biggl[ \frac{1}{\gamma-1}\left( \lp\onabla B_{0}^{2\gamma}(U),\onabla B_{2\gamma-2}^{2\gamma}(V)\rp + \lp\onabla B_{0}^{2\gamma}(V),\onabla B_{2\gamma-2}^{2\gamma}(U)\rp\right) \\
  & \qquad + \frac{n-2\gamma}{2}T_{2}^{2\gamma}\left(B_{0}^{2\gamma}(U)B_{2\gamma-2}^{2\gamma}(V) + B_{0}^{2\gamma}(V)B_{2\gamma-2}^{2\gamma}(U)\right) \\
  & \qquad + \frac{n-2\gamma}{2}\left(\rho^m\eta J_\phi^m\right)B_{0}^{2\gamma}(U)B_{0}^{2\gamma}(V) \biggr] .
 \end{align*}
 In particular, $\mQ_{2\gamma}$ is conformally invariant.
\end{thm}

\begin{proof}
 \eqref{eqn:ibp} follows from a straightforward computation using integration by parts, the consequences
 \begin{equation}
  \label{eqn:gauss_codazzi}
  \begin{split}
   J & = \oJ + P(\eta,\eta) , \\
   P(\eta,\nabla U) & = P(\eta,\eta)\,\eta U
  \end{split}
 \end{equation}
 of the Gauss--Codazzi equations, and~\eqref{eqn:J_to_weight}.
 
 It follows immediately from Theorem~\ref{thm:invariant} and~\eqref{eqn:ibp} that $\mQ_{2\gamma}$ is conformally invariant.
\end{proof}
\section{Some asymptotic expansions}
\label{sec:asymptotics}

For Poincar\'e--Einstein manifolds, the fractional GJMS operator $P_1$ can be interpreted as the Dirichlet-to-Neumann operator for functions in the kernel of the conformal Laplacian~\cite{ChangGonzalez2011}.  There is another conformally covariant operator defined on the boundary with the same principal symbol as $P_1$, namely $f\mapsto B_{1}^{1}U$ for $U$ the unique extension of $f$ in the kernel of the conformal Laplacian~\cite{Branson1997}.  In fact, these two operators are the same~\cite{GuillarmouGuillope2007}.

The boundary operators introduced in Section~\ref{sec:boundary} all give conformally covariant operators as follows: Fix $\gamma\in(0,2)\setminus\{1\}$ and set $k=\lfloor\gamma\rfloor+1$ and $m=2k-1-2\gamma$.  Let $(X^{n+1},M^n,g_+)$ be a Poincar\'e--Einstein manifold with $n>2\gamma$, let $\rho$ be a $\gamma$-admissible defining function, and consider $(\oX,\rho^2g_+,\rho,m,1)$.  Given a function $f\in C^\infty(M)$, let $U$ be the unique extension of $f$ in $\mC_{f,0}^\gamma$ such that $L_{2k,\phi}^mU=0$.  Then the map $\mB_{2\gamma}(f):=B_{2\gamma}^{2\gamma}U$ is conformally covariant in the sense that
\[ \hmB_{2\gamma}(f) = e^{-\frac{n+2\gamma}{2}\sigma\rv_M}\mB_{2\gamma}\left(e^{\frac{n-2\gamma}{2}\sigma\rv_M}f\right) \]
for all $\sigma\in\mD^\gamma$, where $\hmB$ is defined in terms of $(\oX^{n+1},\hg,\hrho,m,1)$ for $\hg=e^{2\sigma}g$ and $\hrho=e^\sigma\rho$.  The fractional GJMS operators can also be regarded as generalized Dirichlet-to-Neumann operators associated to the kernel of the weighted conformal Laplacian and the weighted Paneitz operator in the cases $\gamma\in(0,1)$ and $\gamma\in(1,2)$, respectively~\cite{CaseChang2013}.  We show that, as in the case $\gamma=1/2$, the operators $\mB_{2\gamma}$ and $P_{2\gamma}$ are the same.

\subsection{The case $\gamma\in(0,1)$}
\label{subsec:asymptotics/1}

A direct computation using the definition of $B_{2\gamma}^{2\gamma}$ and~\cite[Theorem~4.1]{CaseChang2013} readily shows that the fractional GJMS operator $P_{2\gamma}$ and the operator $\mB_{2\gamma}$ defined above are the same when $\gamma\in(0,1)$.

\begin{prop}
 \label{prop:extension}
 Fix $\gamma\in(0,1)$ and set $m=1-2\gamma$.  Let $(X^{n+1},M^n,g_+)$ be a Poincar\'e--Einstein manifold such that $\frac{n^2}{4}-\gamma^2\not\in\sigma_{pp}(-\Delta_{g_+})$.  Let $\rho$ be a $\gamma$-admissible defining function.  Given $f\in C^\infty(M)$, let $U$ be the solution to the boundary value problem
 \begin{equation}
  \label{eqn:bvp0}
  \begin{cases}
   L_{2,\phi}^mU = 0, & \text{in $\left(X^{n+1},\rho^2g_+,\rho,m,1\right)$}, \\
   U = f, & \text{on $M$}.
  \end{cases}
 \end{equation}
 Then
 \begin{equation}
 \label{eqn:extension_conclusion}
 P_{2\gamma}f = -\frac{d_\gamma}{2\gamma} B_{2\gamma}^{2\gamma}U .
 \end{equation}
\end{prop}

\begin{proof}
 By conformal covariance, we may assume that $\rho=r$ is the geodesic defining function associated to $g\rv_{TM}$.  From~\cite[Theorem~4.1]{CaseChang2013}, we see that $U=r^{-\frac{n-2\gamma}{2}}\mP(\frac{n}{2}+\gamma)f$.  In particular, $U\in\mC_f^\gamma$ and
 \[ P_{2\gamma}f = -\frac{d_\gamma}{2\gamma}\lim_{\rho\to0}\rho^m\eta U \]
 for $\eta=-\partial_r$ the outward-pointing normal along $M$.  It is straightforward to check that $r^m\delta\eta\to0$ as $r\to0$, and hence~\eqref{eqn:extension_conclusion} holds.
\end{proof}

\subsection{The case $\gamma\in(1,2)$}
\label{subsec:asymptotics/2}

An argument similar to the one given in the proof of Proposition~\ref{prop:extension} shows that $B_{2\gamma}^{2\gamma}U$ is proportional to $P_{2\gamma}f$ if $U\in\mC_{f,0}^\gamma$ satisfies $L_{4,\phi}^mU=0$.  However, much more can be said.  Indeed, this statement remains true if $U\in\mC_{f,\psi}^\gamma$ for any $\psi\in C^\infty(M)$!  To make this more precise, we shall evaluate the operators $B_{s}^{2\gamma}$ for $s\in\{0,2\gamma-2,2,2\gamma\}$ when acting on elements of the kernel of the weighted Paneitz operator in terms of scattering operators.

We begin by investigating the asymptotic behavior of the summands which appear in the definitions of the operators $B_{s}^{2\gamma}$.  For our intended applications, it suffices to compute with respect to the compactification of a Poincar\'e--Einstein manifold by a geodesic defining function.  We first observe the following simple asymptotic behavior of the interior scalar curvature.

\begin{lem}
 \label{lem:eval}
 Fix $\gamma\in(1,2)$ and set $m=3-2\gamma$.  Let $(X^{n+1},M^n,g_+)$ be a Poincar\'e--Einstein manifold and let $r$ be a geodesic defining function.  Then, in terms of $(\oX,r^2g_+,r,m,1)$, it holds that, asymptotically near $M$,
 \begin{equation}
  \label{eqn:eval}
  J = \oJ + O(r^2) .
 \end{equation}
\end{lem}

\begin{proof}
 A straightforward computation shows that
 \[ r^{-1}\Delta r = -\oJ + O(r^2) . \]
 The conclusion now follows from~\eqref{eqn:J_to_grad}.
\end{proof}

We next compute certain derivatives of elements of $\mC^\gamma$.

\begin{lem}
 \label{lem:evalU}
 Fix $\gamma\in(1,2)$ and set $m=3-2\gamma$.  Let $(X^{n+1},M^n,g_+)$ be a Poincar\'e--Einstein manifold and let $r$ be a geodesic defining function.  Let $U\in\mC^\gamma$ have the expansion~\eqref{eqn:mC1f_asympt} asymptotically near $M$.  Then, in terms of $(\oX,r^2g_+,r,m,1)$, it holds that
 \begin{align}
  \label{eqn:eval_eta} \lim_{r\to0} \left[r^m\eta U\right] & = 2(1-\gamma)\psi , \\
  \label{eqn:eval_Hess} \lim_{r\to0} \left[\nabla^2U(\eta,\eta) + mr^{-1}\partial_r U\right] & = 4(2-\gamma)f_2 , \\
  \label{eqn:eval_Deltaphi} \lim_{r\to0} \left[-r^m\eta\Delta_\phi U\right] & = 2(\gamma-1)\Bigl[4\gamma\psi_2 + \oDelta\psi - 2(\gamma-1)\oJ\psi\Bigr] .
 \end{align}
\end{lem}

\begin{proof}
 \eqref{eqn:eval_eta} is an immediate consequence of~\eqref{eqn:mC1f_asympt}.

 We next compute that
 \[ r^{-m}\partial_r\left(r^m\partial_rU\right) = 4(2-\gamma)f_2 + 4\gamma\psi_2r^{2\gamma-2} + o(r^{2\gamma-2}), \]
 from which~\eqref{eqn:eval_Hess} immediately follows.

 Finally,
 \[ \Delta_\phi U = \oDelta f + 4(2-\gamma)f_2 + \left(4\gamma\psi_2 + \oDelta\psi - 2(\gamma-1)\oJ\psi\right)r^{2\gamma-2} + o(r^{2\gamma-2}) . \]
 Differentiating yields~\eqref{eqn:eval_Deltaphi}.
\end{proof}

Combining Lemma~\ref{lem:eval} and Lemma~\ref{lem:evalU} yields the following evaluations of the operators $B_{s}^{2\gamma}$ in terms of the asymptotic expansion~\eqref{eqn:mC1f_asympt}.

\begin{prop}
\label{prop:eval}
 Fix $\gamma\in(1,2)$ and set $m=3-2\gamma$.  Let $(X^{n+1},M^n,g_+)$ be a Poincar\'e--Einstein manifold and let $r$ be a geodesic defining function.  Let $U\in\mC^\gamma$ be such that~\eqref{eqn:mC1f_asympt} holds near $M$.  In terms of $(\oX,r^2g_+,r,m,1)$, it holds that
 \begin{align}
  \label{eqn:eval0} B_{0}^{2\gamma}U & = f, \\
  \label{eqn:eval1} B_{2\gamma-2}^{2\gamma}U & = 2(1-\gamma)\psi, \\
  \label{eqn:eval2} B_{2}^{2\gamma}U & = \frac{2-\gamma}{\gamma-1}\left(-\oDelta f + \frac{n-2\gamma}{2}\oJ f\right) + 4(2-\gamma)f_2 , \\
  \label{eqn:eval3} B_{2\gamma}^{2\gamma}U & = 8\gamma(\gamma-1)\psi_2 - 2\gamma\left(-\oDelta\psi + \frac{n+2\gamma-4}{2}\oJ\psi\right) .
 \end{align}
\end{prop}

\begin{proof}
 \eqref{eqn:eval0} is obvious, while~\eqref{eqn:eval1} is~\eqref{eqn:eval_eta}.  Combining~\eqref{eqn:gauss_codazzi} and~\eqref{eqn:eval} yields that $P(\eta,\eta)=0$.  Hence, by~\eqref{eqn:gauss_codazzi} and Lemma~\ref{lem:eval},
 \begin{equation}
  \label{eqn:asymptotic_T2}
  T_2^{2\gamma} = \frac{2-\gamma}{\gamma-1}\oJ .
 \end{equation}
 It then follows from~\eqref{eqn:eval_Hess} that~\eqref{eqn:eval2} holds.  Finally, Lemma~\ref{lem:eval} implies that
 \[ S_2^{2\gamma} = \left(\frac{n-2\gamma}{2} + \frac{n+2\gamma-4}{2(\gamma-1)}\right)\oJ . \]
 Combining this and Lemma~\ref{lem:evalU} yields~\eqref{eqn:eval3}.
\end{proof}

Applying Proposition~\ref{prop:eval} to solutions of the Poisson equation~\eqref{eqn:poisson} yields the following interpretation of the operators $B_{s}^{2\gamma}$.

\begin{cor}
 \label{cor:eval}
 Fix $\gamma\in(1,2)$ and set $m=3-2\gamma$.  Let $(X^{n+1},M^n,g_+)$ be a Poincar\'e--Einstein manifold such that $\frac{n^2}{4}-\gamma^2,\frac{n^2}{4}-(2-\gamma)^2\not\in\sigma_{pp}(-\Delta_{g_+})$.  Let $\rho$ be a $\gamma$-admissible defining function with expansion~\eqref{eqn:gamma_admissible} near $M$.  Fix $f,\psi\in C^\infty(M)$ and set $u_1=\mP\left(\frac{n}{2}+\gamma\right)f$ and $u_2=\mP\left(\frac{n}{2}+2-\gamma\right)\psi$.  Set $U=\rho^{-\frac{n-2\gamma}{2}}(u_1+u_2)$.  In terms of $(\oX,\rho^2g_+,\rho,m,1)$, it holds that
 \begin{align*}
  L_{4,\phi}^mU & = 0 , \\
  B_{0}^{2\gamma}U & = f, \\
  B_{2\gamma-2}^{2\gamma}U & = 2(1-\gamma)\psi, \\
  B_{2}^{2\gamma}U & = \frac{4(2-\gamma)}{d_{2-\gamma}}P_{4-2\gamma}\psi, \\
  B_{2\gamma}^{2\gamma}U & = \frac{8\gamma(\gamma-1)}{d_\gamma}P_{2\gamma}f .
 \end{align*}
\end{cor}

\begin{proof}
 By conformal covariance, we may assume that $\rho=r$ is the geodesic defining function associated to $g\rv_{TM}$.  That $L_{4,\phi}^mU=0$ follows from conformal covariance and the factorization
 \begin{equation}
  \label{eqn:paneitz_factorization}
  \left(L_{4,\phi}^m\right)_{g_+} = \left(-\Delta_{g_+}-\frac{n^2}{4}+(2-\gamma)^2\right)\circ\left(-\Delta_{g_+}-\frac{n^2}{4}+\gamma^2\right)
 \end{equation}
 of the weighted Paneitz operator of $(X^{n+1},g_+,1,m,1)$; see~\cite[Theorem~3.1]{CaseChang2013}.

 Using the asymptotic expansion~\eqref{eqn:Fexpansion2}, we observe that
 \begin{align*}
  u_1 & \cong r^{\frac{n-2\gamma}{2}}\left[f + \frac{1}{4(1-\gamma)}\left(-\oDelta f + \frac{n-2\gamma}{2}\oJ f\right)r^2 + d_\gamma^{-1}P_{2\gamma}f r^{2\gamma}\right] , \\
  u_2 & \cong r^{\frac{n-4+2\gamma}{2}}\Bigl[ \psi + d_{2-\gamma}^{-1}P_{4-2\gamma}(\psi)r^{4-2\gamma} + \frac{1}{4(\gamma-1)}\left(-\oDelta\psi + \frac{n+2\gamma-4}{2}\oJ\psi\right)r^2\Bigr] ,
 \end{align*}
 where we write $A\cong B$ if $A-B=o(r^{\frac{n+2\gamma}{2}})$.  Set $U_j=r^{-\frac{n-2\gamma}{2}}u_j$ for $j\in\{1,2\}$.  Applying Proposition~\ref{prop:eval} to $U_1$ and $U_2$ and using the linearity of the operators $B_{s}^{2\gamma}$ for $s\in\{0,2\gamma-2,2,2\gamma\}$ yields the result.
\end{proof}

By appealing to the uniqueness of solutions to $L_{4,\phi}^mU=0$ with suitable boundary conditions, we obtain two extension theorems relating the fractional GJMS operator $P_{2\gamma}$ to the operator $B_{2\gamma}^{2\gamma}$.  The first result is a reformulation of~\cite[Theorem~4.4]{CaseChang2013}.

\begin{prop}
 \label{prop:eval_mixed}
 Fix $\gamma\in(1,2)$ and set $m=3-2\gamma$.  Let $(X^{n+1},M^n,g_+)$ be a Poincar\'e--Einstein manifold such that $\frac{n^2}{4}-\gamma^2,\frac{n^2}{4}-(2-\gamma)^2\not\in\sigma_{pp}(-\Delta_{g_+})$.  Let $\rho$ be a $\gamma$-admissible defining function.  Given $f\in C^\infty(M)$, let $U$ be the unique solution to the boundary value problem
 \begin{equation}
  \label{eqn:bvp_mixed}
  \begin{cases}
   L_{4,\phi}^mU = 0, & \text{in $\left(\oX^{n+1},\rho^2g_+,\rho,m,1\right)$} \\
   B_{0}^{2\gamma}U = f, & \text{on $M$} \\
   B_{2\gamma-2}^{2\gamma}U = 0, & \text{on $M$}.
  \end{cases}
 \end{equation}
 Then
 \begin{equation}
 \label{eqn:eval_mixed}
  P_{2\gamma}f = \frac{d_\gamma}{8\gamma(\gamma-1)}B_{2\gamma}^{2\gamma}U .
 \end{equation}
\end{prop}

\begin{proof}
 Let $u=\mP\left(\frac{n}{2}+\gamma\right)f$ and set $\tilde U=\rho^{-\frac{n-2\gamma}{2}}u$.  It follows from Proposition~\ref{prop:eval} and conformal covariance that $\tilde U$ satisfies~\eqref{eqn:bvp_mixed}.  Hence, by uniqueness of solutions of~\eqref{eqn:bvp_mixed}, it holds that $U=\tilde U$.  \eqref{eqn:eval_mixed} now follows from Corollary~\ref{cor:eval}.
\end{proof}

The second result is an analogous extension theorem formulated in terms of the iterated Dirichlet data of a fourth order boundary value problem.

\begin{prop}
 \label{prop:eval_dirichlet}
 Fix $\gamma\in(1,2)$ and set $m=3-2\gamma$.  Let $(X^{n+1},M^n,g_+)$ be a Poincar\'e--Einstein manifold such that $\frac{n^2}{4}-\gamma^2,\frac{n^2}{4}-(2-\gamma)^2\not\in\sigma_{pp}(-\Delta_{g_+})$.  Let $\rho$ be a $\gamma$-admissible defining function.  Given $f\in C^\infty(M)$, let $U$ be the unique solution to the boundary value problem
 \begin{equation}
  \label{eqn:bvp_dirichlet}
  \begin{cases}
   L_{4,\phi}^mU = 0, & \text{in $\left(\oX^{n+1},\rho^2g_+,\rho,m,1\right)$} \\
   B_{0}^{2\gamma}U = f, & \text{on $M$} \\
   B_{2}^{2\gamma}U = 0, & \text{on $M$}.
  \end{cases}
 \end{equation}
 Then
 \begin{equation}
 \label{eqn:eval_dirichlet}
  P_{2\gamma}f = \frac{d_\gamma}{8\gamma(\gamma-1)}B_{2\gamma}^{2\gamma}U .
 \end{equation}
\end{prop}

\begin{proof}
 Let $u=\mP\left(\frac{n}{2}+\gamma\right)f$ and set $\tilde U=\rho^{-\frac{n-2\gamma}{2}}u$.  It follows from Proposition~\ref{prop:eval} and conformal covariance that $\tilde U$ satisfies~\eqref{eqn:bvp_dirichlet}.  Hence, by uniqueness of solutions of~\eqref{eqn:bvp_dirichlet}, it holds that $U=\tilde U$.  \eqref{eqn:eval_dirichlet} now follows from Corollary~\ref{cor:eval}.
\end{proof}

Surprisingly, the fractional GJMS operator $P_{2\gamma}$ can be recovered from the boundary operator $B_{2\gamma}^{2\gamma}$ without finding a unique extension.

\begin{prop}
 \label{prop:eval_surprise}
 Fix $\gamma\in(1,2)$ and set $m=3-2\gamma$.  Let $(X^{n+1},M^n,g_+)$ be a Poincar\'e--Einstein manifold such that $\frac{n^2}{4}-\gamma^2,\frac{n^2}{4}-(2-\gamma)^2\not\in\sigma_{pp}(-\Delta_{g_+})$.  Let $\rho$ be a $\gamma$-admissible defining function.  Suppose that $U\in\mC^\gamma$ satisfies
 \[ \begin{cases}
     L_{4,\phi}^mU = 0, & \text{in $\left(\oX^{n+1},\rho^2g_+,\rho,m,1\right)$} \\
     B_{0}^{2\gamma}U = f, & \text{on $M$}
    \end{cases} \]
 Then
 \begin{equation}
  \label{eqn:eval_surprise}
  P_{2\gamma}f = \frac{d_\gamma}{8\gamma(\gamma-1)}B_{2\gamma}^{2\gamma}U .
 \end{equation}
\end{prop}

\begin{proof}
 Set $\psi=\frac{1}{2(1-\gamma)}B_{2\gamma-2}^{2\gamma}U$.  Let $u_1=\mP\left(\frac{n}{2}+\gamma\right)f$ and $u_2=\mP\left(\frac{n}{2}+2-\gamma\right)\psi$ and set $\tilde U=\rho^{-\frac{n-2\gamma}{2}}(u_1+u_2)$.  It follows from Corollary~\ref{cor:eval} that $\tilde U$ satisfies
 \begin{equation}
  \label{eqn:bvp_surprise}
  \begin{cases}
   L_{4,\phi}^m\tilde U = 0, & \text{in $X$} \\
   B_{0}^{2\gamma}\tilde U = f, & \text{on $M$} \\
   B_{2\gamma-2}^{2\gamma}\tilde U = 2(1-\gamma)\psi, & \text{on $M$}.
  \end{cases}
 \end{equation}
 By uniqueness of solutions of~\eqref{eqn:bvp_surprise}, it holds that $U=\tilde U$.  \eqref{eqn:eval_surprise} now follows from Corollary~\ref{cor:eval}.
\end{proof}
\section{The energy inequality}
\label{sec:inequality}

We are now ready to prove the energy inequalities stated in the introduction.  As throughout this article, we consider the cases $\gamma\in(0,1)$ and $\gamma\in(1,2)$ separately.  Nevertheless, the basic ideas are the same.  We start by considering the energy functional $\mE_{2\gamma}\colon C^\gamma\to\bR$ given by $\mE_{2\gamma}(U)=\mQ_{2\gamma}(U,U)$, where $\mQ_{2\gamma}$ is as in Theorem~\ref{thm:robin} and Theorem~\ref{thm:integral}.  Concretely, if $\gamma\in(0,1)$, then
\begin{equation}
 \label{eqn:mE0}
 \mE_{2\gamma}(U) = \int_X \left(\lv\nabla U\rv^2 + \frac{n-2\gamma}{2}J_\phi^mU^2\right) + \frac{n-2\gamma}{2n}\oint_M H_{2\gamma}\left(B_{0}^{2\gamma}U\right)^2 ,
\end{equation}
while if $\gamma\in(1,2)$, then
\begin{equation}
 \label{eqn:mE1}
 \begin{split}
  \mE_{2\gamma}(U) & = \int_X \left(\left(\Delta_\phi U\right)^2 - (4P-(n-2\gamma+2)J_\phi^mg)(\nabla U,\nabla U) + \frac{n-2\gamma}{2}Q_\phi^mU^2\right) \\
  & \quad + \oint_M \biggl(\frac{2}{\gamma-1}\lp\onabla B_{0}^{2\gamma}(U),\onabla B_{2\gamma-2}^{2\gamma}(U)\rp \\
  & \qquad + (n-2\gamma)T_{2}^{2\gamma}B_{0}^{2\gamma}(U)B_{2\gamma-2}^{2\gamma}(U) + \frac{n-2\gamma}{2}(\rho^m\eta J_\phi^m)\left(B_{0}^{2\gamma}U\right)^2\biggr) .
 \end{split}
\end{equation}
Using the fact that $\mC_{f,\psi}^\gamma = U + \mC_{0,0}^\gamma$ for some, and hence any, fixed $U\in\mC_{f,\psi}^\gamma$, we obtain a necessary and sufficient condition for $\mE_{2\gamma}$ to be uniformly bounded below in $\mC_{f,\psi}^\gamma$ in terms of the bottom of the $L^2$-spectrum of $L_{2k,\phi}^m$ for $k=\lfloor\gamma\rfloor+1$.  When $\mE_{2\gamma}$ is uniformly bounded below, we construct a minimizer which necessarily solves~\eqref{eqn:bvp0} or~\eqref{eqn:bvp_mixed}.  This construction together with Proposition~\ref{prop:extension} or Proposition~\ref{prop:eval_mixed}, respectively, yields the result.

\subsection{The case $\gamma\in(0,1)$}
\label{subsec:inequality/1}

Following the outline above, we start by characterizing when $\mE_{2\gamma}$ is uniformly bounded below on $\mC_f^\gamma$ in terms of the bottom of the spectrum of $-\Delta_{g_+}$ on a Poincar\'e--Einstein manifold $(X^{n+1},M^n,g_+)$.  This result generalizes an observation of Escobar~\cite{Escobar1992ac} in the case $\gamma=1/2$, and corrects an error in the remark following the statement of~\cite[Theorem~1.4]{GonzalezQing2010}.

\begin{prop}
 \label{prop:bounded_below}
 Fix $\gamma\in(0,1)$ and set $m=1-2\gamma$.  Let $(X^{n+1},M^n,g_+)$ be a Poincar\'e--Einstein manifold such that $\frac{n^2}{4}-\gamma^2\not\in\sigma_{pp}(-\Delta_{g_+})$.  Let $\rho$ be a $\gamma$-admissible defining function and consider $(\oX,\rho^2g_+,\rho,m,1)$.  Fix $f\in C^\infty(M)$.  Then
 \begin{equation}
  \label{eqn:finite_energy}
  \inf_{U\in\mC_f^\gamma} \mE_{2\gamma}(U) > -\infty
 \end{equation}
 if and only if
 \begin{equation}
  \label{eqn:spectrum}
  \lambda_1\left(-\Delta_{g_+}\right) > \frac{n^2}{4} - \gamma^2 .
 \end{equation}
\end{prop}

\begin{proof}
 Fix $U\in\mC_f^\gamma$, so that $\mC_f^\gamma=U+\mC_0^\gamma$.  By definition,
 \[ \lambda_1\left(L_{2,\phi}^m\right) = \inf \left\{ \mE_{2\gamma}(V) \suchthat V \in \mC_0^\gamma, \int_X V^2=1 \right\} . \]
 Since $(L_{2,\phi}^m)_+=-\Delta_{g_+}-\frac{n^2}{4}+\gamma^2$, we have that $\lambda_1(L_{2,\phi}^m)>0$ if and only if~\eqref{eqn:spectrum} holds.
 
 Next, given any $V\in\mC_0^\gamma$ and $t\in\bR$, we compute that
 \begin{equation}
  \label{eqn:basic_energy_estimate0}
  \mE_{2\gamma}(U+tV) = t^2\mE_{2\gamma}(V) + 2t\mQ_{2\gamma}(U,V) + \mE_{2\gamma}(U) .
 \end{equation}
 If~\eqref{eqn:spectrum} does not hold, then there is a $V\in\mC_0^\gamma$ such that $\mE_{2\gamma}(V)<0$.  In particular, inserting this $V$ into~\eqref{eqn:basic_energy_estimate0} and letting $t\to\infty$ shows that~\eqref{eqn:finite_energy} does not hold.  If~\eqref{eqn:spectrum} holds, then Theorem~\ref{thm:robin} and~\eqref{eqn:basic_energy_estimate0} imply that
 \[ \mE_{2\gamma}(U+V) \geq \lambda_1\left(L_{2,\phi}^m\right)\int_X V^2 + 2\left(\int_X \left(L_{2,\phi}^mU\right)^2\right)^{\frac{1}{2}}\left(\int_X V^2\right)^{\frac{1}{2}} + \mE_{2\gamma}(U) \]
 for any $V\in\mC_0^\gamma$.  An application of the Cauchy--Schwarz inequality yields a lower bound for $\mE_{2\gamma}(U+V)$ which depends only on $U$.  In particular, \eqref{eqn:finite_energy} holds.
\end{proof}

We next implement the minimization scheme described at the beginning of this section to prove Theorem~\ref{thm:intro_inequality}.  For convenience, we restate the result here.

\begin{thm}
\label{thm:inequality}
 Fix $\gamma\in(0,1)$ and set $m=1-2\gamma$.  Let $(X^{n+1},M^n,g_+)$ be a Poincar\'e--Einstein manifold which satisfies~\eqref{eqn:spectrum}.  Let $\rho$ be a $\gamma$-admissible defining function with expansion~\eqref{eqn:gamma_admissible} and consider $(\oX,\rho^2g_+,\rho,m,1)$.  For any $f\in C^\infty(M)$, it holds that
 \begin{multline}
  \label{eqn:inequality}
  \int_X \left(\lv\nabla U\rv^2 + \frac{m+n-1}{2}J_\phi^m U^2\right) \rho^m\,\dvol \\ \geq -\frac{2\gamma}{d_\gamma}\left[ \oint_M f\,P_{2\gamma}f\,\dvol - \frac{n-2\gamma}{2}d_\gamma\oint_M\Phi f^2\dvol \right] .
 \end{multline}
 for all $U\in W^{1,2}(\oX,\rho^m\dvol)$ with $\Tr U=f$.  Moreover, equality holds if and only if $L_{2,\phi}^mU=0$.
\end{thm}

\begin{proof}
 From~\eqref{eqn:mE0} and Proposition~\ref{prop:bounded_below} we observe that the left-hand side of~\eqref{eqn:inequality} is uniformly bounded below in $\mC_f^\gamma$.  Let $U\in W^{1,2}(\oX,\rho^m\dvol)$ be a minimizer.  Necessarily $U$ solves~\eqref{eqn:bvp0}.  Theorem~\ref{thm:robin} and Proposition~\ref{prop:extension} then imply that
 \[ \mE_{2\gamma}(U) = -\frac{2\gamma}{d_\gamma}\oint_M f\,P_{2\gamma}f, \]
 from which~\eqref{eqn:inequality} immediately follows.
\end{proof}

\subsection{The case $\gamma\in(1,2)$}
\label{subsec:inequality/2}

We again start by finding a necessary and sufficient condition for the energy $\mE_{2\gamma}$ to be uniformly bounded below on $\mC_{f,\psi}^\gamma$.

\begin{prop}
\label{prop:bounded_below2}
 Fix $\gamma\in(1,2)$ and set $m=3-2\gamma$.  Let $(X^{n+1},M^n,g_+)$ be a Poincar\'e--Einstein manifold such that $\frac{n^2}{4}-\gamma^2,\frac{n^2}{4}-(2-\gamma)^2\not\in\sigma_{pp}(-\Delta_{g_+})$.  Let $\rho$ be a $\gamma$-admissible defining function and consider $(\oX,\rho^2g_+,\rho,m,1)$.  Fix $f,\psi\in C^\infty(M)$.  Then
 \begin{equation}
  \label{eqn:finite_energy2}
  \inf_{U\in\mC^\gamma_{f,\psi}} \mE_\gamma(U) > -\infty
 \end{equation}
 if and only if
 \begin{equation}
  \label{eqn:spectrum2}
  \lambda_1\left(L_{4,\phi}^m\right) > 0 .
 \end{equation}
 Moreover, if $\lambda_1(-\Delta_{g_+})>\frac{n^2}{4}-(2-\gamma)^2$, then~\eqref{eqn:spectrum2} holds.
\end{prop}

\begin{proof}
 Arguing as in the proof of Proposition~\ref{prop:bounded_below}, but using Theorem~\ref{thm:integral} instead of Theorem~\ref{thm:robin}, yields the equivalence of~\eqref{eqn:finite_energy2} and~\eqref{eqn:spectrum2}.
 
 Suppose now that $\lambda_1(-\Delta_{g_+})>\frac{n^2}{4}-(2-\gamma)^2$.  It follows immediately from~\eqref{eqn:paneitz_factorization} that~\eqref{eqn:spectrum2} holds.
\end{proof}

We next implement the minimization scheme to derive the following improvement of Theorem~\ref{thm:intro_inequality2}.

\begin{thm}
\label{thm:inequality2}
 Fix $\gamma\in(1,2)$ and set $m=3-2\gamma$.  Let $(X^{n+1},M^n,g_+)$ be a Poincar\'e--Einstein manifold which satisfies~\eqref{eqn:spectrum2}.  Let $\rho$ be a $\gamma$-admissible defining function with expansion~\eqref{eqn:gamma_admissible} and consider $(\oX,\rho^2g_+,\rho,m,1)$.  For any $f,\psi\in C^\infty(M)$, it holds that
 \begin{equation}
  \label{eqn:inequality2}
  \begin{split}
   & \int_X\left(\left(\Delta_\phi U\right)^2 - \left(4P-(n-2\gamma+2)J_\phi^mg\right)(\nabla U,\nabla U) + \frac{n-2\gamma}{2}Q_\phi^m U^2\right) \\
   & \quad \geq \frac{8\gamma(\gamma-1)}{d_\gamma}\left(\oint_M f\,P_{2\gamma}f - \frac{n-2\gamma}{2}d_\gamma\oint_M \Phi f^2\right) \\
   & \qquad + \frac{d_\gamma}{2\gamma(\gamma-1)}\oint_M \psi\,P_{4-2\gamma}\psi \\
   & \qquad + 4\oint_M \left( \lp\onabla f,\onabla\psi\rp + \frac{(n-2\gamma)(2-\gamma)}{2}\left(\oJ+4(\gamma-1)\rho_2\right)f\psi\right)
  \end{split}
 \end{equation}
 for all $U\in W^{2,2}(\oX,\rho^m\dvol)$ with $\Tr U=(f,\psi)$.  Moreover, equality holds if and only if $L_{4,\phi}^mU=0$.
\end{thm}

\begin{proof}
 From~\eqref{eqn:eval1}, \eqref{eqn:mE1} and Proposition~\ref{prop:bounded_below2}, we observe that the left-hand side of~\eqref{eqn:inequality2} is uniformly bounded below in $\mC_{f,\psi}^\gamma$.  Let $U\in W^{2,2}(\oX,\rho^m\dvol)$ be a minimizer.  Necessarily $U$ solves~\eqref{eqn:bvp_surprise}.  Theorem~\ref{thm:integral} and Corollary~\ref{cor:eval} then imply that
 \[ \mE_{2\gamma}(U) = \frac{8\gamma(\gamma-1)}{d_\gamma}\oint_M f\,P_{2\gamma}f + 8(1-\gamma)(2-\gamma)d_{2-\gamma}^{-1}\oint_M \psi\,P_{4-2\gamma}\psi . \]
 The conclusion follows from~\eqref{eqn:fractional_gjms}.
\end{proof}

\section{A sharp Sobolev inequality}
\label{sec:sobolev}

As an application of Theorem~\ref{thm:inequality2}, we prove the following sharp Sobolev trace inequality for $\gamma$-admissible compactifications of hyperbolic space with $\gamma\in(1,2)$.  This statement is more general than Theorem~\ref{thm:sobolev} in that it involves the full trace on $\mC^\gamma$ and it allows for arbitrary $\gamma$-admissible compactifications.

\begin{thm}
 \label{thm:sobolev_genl}
 Fix $\gamma\in(1,2)$ and set $m=3-2\gamma$.  Choose $n\in\bN$ such that $n>2\gamma$ and let $\rho$ be a $\gamma$-admissible defining function on hyperbolic space $(H^{n+1},S^n,g_+)$.  Then, in terms of $(\overline{H},\rho^2g_+,\rho,m,1)$,
 \[ \mE_{2\gamma}(U) \geq c_{n,\gamma}^{(2)}\left(\oint_{S^n} \lv f\rv^{\frac{2n}{n-2\gamma}}\right)^{\frac{n-2\gamma}{n}} + \frac{1}{4}c_{n,2-\gamma}^{(2)}\left(\oint_{S^n} \lv \psi\rv^{\frac{2n}{n-4+2\gamma}}\right)^{\frac{n-4+2\gamma}{n}} \]
 for all $U\in\mC^\gamma$, where $f$ is the trace of $U$, $\Phi$ is as in~\eqref{eqn:gamma_admissible}, and
 \[ c_{n,\gamma}^{(2)} = 8\pi^\gamma\frac{\Gamma(2-\gamma)}{\Gamma(\gamma)}\frac{\Gamma\left(\frac{n+2\gamma}{2}\right)}{\Gamma\left(\frac{n-2\gamma}{2}\right)}\left(\frac{\Gamma(n/2)}{\Gamma(n)}\right)^{\frac{2\gamma}{n}} . \]
 Moreover, equality holds if and only if $L_{4,\phi}^mU=0$ and both $f^{\frac{4}{n-2\gamma}}(\rho^2g_+)\rv_{TS^n}$ and $\psi^{\frac{4}{n-4+2\gamma}}(\rho^2g_+)\rv_{TS^n}$ are Einstein with positive scalar curvature.
\end{thm}

\begin{proof}
 Since $\lambda_1(-\Delta_{g_+})=\frac{n^2}{4}$, Theorem~\ref{thm:inequality2} implies that
 \begin{equation}
  \label{eqn:sobolev_preinequality}
  \mE_{2\gamma}(U) \geq \frac{8\gamma(\gamma-1)}{d_\gamma}\oint_{S^n} f\,P_{2\gamma}f + \frac{d_\gamma}{2\gamma(\gamma-1)}\oint_{S^n} \psi\,P_{4-2\gamma}\psi
 \end{equation}
 for all $U\in\mC^\gamma$, where $f$ and $\psi$ are as in~\eqref{eqn:mC1f_asympt}.  The conformal covariance of the fractional GJMS operators and the sharp fractional Sobolev inequality~\cite{Beckner1993,CotsiolisTavoularis2004,FrankLieb2012b,Lieb1983} together imply that
 \begin{equation}
  \label{eqn:lieb}
  \oint_{S^n}f\,P_{2\gamma}f \geq 2^{2\gamma}\pi^\gamma\frac{\Gamma\left(\frac{n+2\gamma}{2}\right)}{\left(\frac{n-2\gamma}{2}\right)}\left(\frac{\Gamma(n/2)}{\Gamma(n)}\right)^{\frac{2\gamma}{n}}\left(\oint_{S^n} \lv f\rv^{\frac{2n}{n-2\gamma}}\right)^{\frac{n-2\gamma}{n}}
 \end{equation}
 with equality if and only if $f^{\frac{4n}{n-2\gamma}}(\rho^2g_+)\rv_{TS^n}$ is Einstein with positive scalar curvature.  Combining~\eqref{eqn:sobolev_preinequality}, \eqref{eqn:lieb} and the corresponding result for $P_{4-2\gamma}$ yields the desired result.
\end{proof}

\begin{proof}[Proof of Theorem~\ref{thm:sobolev}]
 It is straightforward to check that $(S_+^{n+1},d\theta^2,x_{n+1},m,1)$ satisfies
 \[ L_{4,\phi}^m = \left(-\Delta_\phi + \frac{(m+n)^2-1}{4}\right)\left(-\Delta_\phi + \frac{(m+n)^2-9}{4}\right) . \]
 Since $g_+=x_{n+1}^{-2}d\theta^2$ satisfies $\Ric_{g_+}=-ng_+$, we see find that $x_{n+1}$ is a $\gamma$-admissible defining function for hyperbolic space.  Applying Theorem~\ref{thm:sobolev_genl} then yields the desired conclusion.
\end{proof}

\appendix
\section{Proof of the Sobolev Trace Theorem}
\label{sec:appendix}

The proof of Theorem~\ref{thm:inequality2} requires the Sobolev space $W^{2,2}(\oX^{n+1},\rho^m\dvol)$ and its trace onto $H^\gamma(M)\oplus H^{2-\gamma}(M)$.  Since our definitions of $W^{2,2}(\oX^{n+1},\rho^m\dvol)$ and the trace via the space $\mC^\gamma$ are nonstandard --- the usual approach is via completions of $C^\infty(\oX)$~\cite{Triebel1978} --- we prove the existence of the trace map.

\begin{thm}
 \label{thm:trace}
 Fix $\gamma\in(1,2)$ and set $m=3-2\gamma$.  Let $(\oX^{n+1},g,\rho,m,1)$ be a $\gamma$-admissible smooth metric measure space with nonempty boundary $M=\partial X$.  There is a unique bounded linear operator
 \[ \Tr\colon W^{2,2}(\oX,\rho^m\dvol) \to H^\gamma(M) \oplus H^{2-\gamma}(M) \]
 such that $\Tr(U)=(f,\psi)$ for all $U\in\mC_{f,\psi}^\gamma$.  Moreover, there is a continuous mapping
 \[ E \colon H^\gamma(M) \oplus H^{2-\gamma}(M) \to W^{2,2}(\oX,\rho^m\dvol) \]
 such that $\Tr\circ E$ is the identity map on $H^\gamma(M)\oplus H^{2-\gamma}(M)$.
\end{thm}

The proof of Theorem~\ref{thm:trace} involves two steps.  First, we prove the corresponding result in upper half space $(\overline{\bR_+^{n+1}},dy^2\oplus dx^2,y,m,1)$ for $\bR_+^{n+1}=(0,\infty)\times\bR^n$ and $y$ the standard coordinate on $(0,\infty)$.  Second, we use coordinate charts to pull the Euclidean result back to $\gamma$-admissible smooth metric measure spaces.  Indeed, the second step is routine, and the proof will be omitted.  The first step is carried out via the extension theorems for the fractional Laplacian~\cite{CaffarelliSilvestre2007,Yang2013}.  To that end, define $\mC_{f,\psi}^\gamma$ and $\mC^\gamma$ as in Subsection~\ref{subsec:bg/smms} using Schwartz functions to obtain the completion $W^{2,2}(\bR_+^{n+1},y^m\dvol)$, and recall that the $H^\gamma$- and $H^{2-\gamma}$-norms defined in Subsection~\ref{subsec:bg/smms} are equivalent to the ones defined via Fourier transform.

\begin{thm}
 \label{thm:trace_rn}
 Fix $\gamma\in(1,2)$ and set $m=3-2\gamma$.  There is a unique bounded linear operator
 \begin{equation}
  \label{eqn:trace_mapping}
  \Tr\colon W^{2,2}(\bR_+^{n+1},y^m\dvol) \to H^\gamma(\bR^n) \oplus H^{2-\gamma}(\bR^n)
 \end{equation}
 such that $\Tr(U)=(f,\psi)$ for any $U\in\mC_{f,\psi}^\gamma$.  Moreover, there is a bounded linear operator
 \begin{equation}
  \label{eqn:E_ext}
  E\colon H^\gamma(\bR^n) \oplus H^{2-\gamma}(\bR^n) \to W^{2,2}(\bR_+^{n+1},y^m\dvol)
 \end{equation}
 such that $\Tr\circ E$ is the identity map on $H^\gamma(\bR^n)\oplus H^{2-\gamma}(\bR^n)$.
\end{thm}

\begin{proof}
 Combining Theorem~\ref{thm:inequality2} and the proof of~\cite[Theorem~3.1]{Yang2013}, we find that for any $U\in\mC_{f,\psi}^\gamma$, it holds that
 \begin{multline}
  \label{eqn:preenergy}
  \int_{\bR_+^{n+1}}\left(\Delta_\phi U\right)^2 \geq \frac{8\gamma(\gamma-1)}{d_\gamma}\lV (-\Delta)^{\gamma/2}f\rV_2^2 \\ + 4\oint_{\bR^n}\lp\onabla f,\onabla\psi\rp + \frac{d_\gamma}{2\gamma(\gamma-1)}\lV(-\Delta)^{(2-\gamma)/2}\psi\rV_2^2
 \end{multline}
 with equality if and only if $U$ is the unique solution to
 \begin{equation}
  \label{eqn:ext}
  \begin{cases}
   \Delta_\phi^2U = 0, & \text{in $(0,\infty)\times\bR^n$}, \\
   U(0,x) = f(x), & \text{for all $x\in\bR^n$}, \\
   \displaystyle\lim_{y\to0} y^m\partial_yU(y,x) = 2(\gamma-1)\psi, & \text{for all $x\in\bR^n$}
  \end{cases}
 \end{equation}
 with $\Delta_\phi=\Delta+my^{-1}\partial_y$.  Using integration by parts, it is straightforward to check that
 \[ \int_{\bR_+^{n+1}} \left|\nabla^2U + my^{-1}(\partial_yU)^2dy\otimes dy\right|^2 = \int_{\bR_+^{n+1}} \left(\Delta_\phi U\right)^2 - 4(\gamma-1)\oint_{\bR^n}\lp\onabla f,\onabla\psi\rp \]
 for all $U\in\mC_{f,\psi}^\gamma$.  Combining this with~\eqref{eqn:preenergy} yields
 \begin{multline}
  \label{eqn:energy}
  \int_{\bR_+^{n+1}} \left|\nabla^2U + my^{-1}(\partial_yU)^2dy\otimes dy\right|^2 \\ \geq \frac{8\gamma(\gamma-1)^2}{d_\gamma}\left\lV(-\Delta)^{\gamma/2}f\right\rV_2^2 + \frac{d_\gamma}{2\gamma}\left\lV(-\Delta)^{(2-\gamma)/2}\psi\right\rV_2^2
 \end{multline}
 with the same characterization of the equality case.  It follows from~\eqref{eqn:energy} that $\Tr\colon\mC^\gamma\to H^\gamma\oplus H^{2-\gamma}$ is a bounded linear operator, and hence can be extended uniquely to a bounded linear operator as in~\eqref{eqn:trace_mapping}.  Moreover, the map $E(f,\psi)=U$ obtained by solving~\eqref{eqn:ext} is linear and, by~\eqref{eqn:energy}, bounded, whence can be extended to a bounded linear operator as in~\eqref{eqn:E_ext}.
\end{proof}

\bibliographystyle{abbrv}
\bibliography{../bib}
\end{document}